\documentclass[12pt,a4paper]{article}
\usepackage[utf8]{inputenc}
\usepackage[T1]{fontenc}
\usepackage{amsmath}
\usepackage{amsfonts}
\usepackage{amssymb}
\usepackage{amsthm}
\usepackage{mathtools}
\usepackage{hyperref}

\usepackage[top=2cm, bottom=2cm, left=2.5cm, right=2.5cm]{geometry}
\newtheorem{thm}{Theorem}[section]
\newtheorem{lem}[thm]{Lemma}
\newtheorem{prop}[thm]{Proposition}
\newtheorem{cor}[thm]{Corollary}

\theoremstyle{definition}
\newtheorem{defn}{Definition}[section]

\theoremstyle{remark}
\newtheorem{rem}{Remark}[section]

\newcommand\Tr{\operatorname{Tr}}

\newcommand\Rot{\operatorname{Rot}}
\newcommand\Lift{\operatorname{Lift}}
\newcommand\Li{\operatorname{Li}}

\allowdisplaybreaks

\begin{document}
\hypersetup{pageanchor=false}
\title{Loop clusters on the discrete circle}
\author{Yinshan Chang\\
\small{Max Planck Institute for Mathematics in the Sciences,}\\
\small{04103 Leipzig, Germany}\\
\small{ychang@mis.mpg.de}}
\date{\vspace{-5ex}}
\maketitle

\begin{abstract}
The loop clusters of a Poissonian ensemble of Markov loops on a finite or countable graph have been studied in \cite{Markovian-loop-clusters-on-graphs}. In the present article, we study the loop clusters associated with a rotation invariant nearest neighbor walk on the discrete circle $G^{(n)}$ with $n$ vertices. We prove a convergence result of the loop clusters on $G^{(n)}$, as $n\rightarrow\infty$, under suitable condition of the parameters. These parameters are chosen in such a way that the rotation invariant nearest neighbor walk on $G^{(n)}$, as $n\rightarrow\infty$, converges to a Brownian motion on circle $\mathbb{S}^{1}=\mathbb{R}/\mathbb{Z}$ with certain drift and killing rate. In the final section, we show that several limit results are predicted by Brownian loop-soup on $\mathbb{S}^{1}$.
\end{abstract}

\section{Introduction}
The loop cluster model is a model of random graphs constructed from a Poisson point process of loops on a finite or countable graph. An edge is defined to be open iff it is crossed by at least one loop in the Poisson point process. Then the open edges form loop clusters. The intensity measure of the Poisson point process is determined by some Markov chain on the same graph. This model was introduced and studied by Y. Le Jan in \cite{Amas-de-lacets-markoviens} and then by S. Lemaire and Y. Le Jan in \cite{Markovian-loop-clusters-on-graphs}. As an example in \cite{Markovian-loop-clusters-on-graphs}, they considered the loop cluster model associated with simple random walks on $\mathbb{Z}$ with uniform killing measures. In the present paper, we study the following variant: the loop cluster models associated with rotation-invariant nearest neighbor walks on discrete circles. 

\subsection{Basic settings}\label{subsect: basic settings}

For simplicity, we denote by Model$(\alpha,n,p_n,c_n)$ the model described in the following:

Consider a discrete circle $G^{(n)}$ with $n$ vertices $1,\ldots,n$ and $2n$ directed edges $$E^{(n)}=\{(1,2),(2,3),\ldots,(n-1,n),(n,1),(2,1),(3,2),\ldots,(n,n-1),(1,n)\}.$$
Define the clockwise edges set $E^{(n)}_{+}=\{(1,2),(2,3),\ldots,(n-1,n),(n,1)\}$ and the counter clockwise edges set $E^{(n)}_{-}=E^{(n)}\setminus E^{(n)}_{+}$. Consider a (sub-)Markovian generator\footnote{See \cite[Definition 2.1]{yinshan} for a precise definition of (sub-)Markovian generator.} $L^{(n)}$ which is the following matrix:
\begin{itemize}
 \item for any $e=(e-,e+)\in E^{(n)}_{+}$, $(L^{(n)})^{e-}_{e+}=p_n,(L^{(n)})^{e+}_{e-}=1-p_n,(L^{(n)})^{e-}_{e-}=-(1+c_n)$ for some numbers $0<p_n<1$ and $c_n>0$,
 \item $L^{(n)}$ is null elsewhere.
\end{itemize}
Set $(Q^{(n)})^{x}_{y}=1_{\{x\neq y\}}\frac{(L^{(n)})^x_y}{-(L^{(n)})^x_x}$. As in \cite{loop} and \cite{SznitmanMR2932978}, we define a loop measure and a Poissonian loop ensemble associated with $L^{(n)}$. By a pointed loop $\dot{\ell}=(x_1,\ldots,x_k)$ of length $k$, we mean a bridge on the graph from $x_1$ back to itself: $x_1\rightarrow x_2\rightarrow\cdots\rightarrow x_k\rightarrow x_1$. It is called ``non-trivial'' when $k\geq 2$. We define the non-trivial pointed loop measure by defining the mass of a pointed loop $\ell=(x_1,\ldots,x_k)$ ($k\geq 2$) as follows:
\begin{equation}\label{eq: 1}
\dot{\mu_n}(\dot{\ell}=(x_1,\ldots,x_k))=\frac{1}{k}(Q^{(n)})^{x_1}_{x_2}\cdots (Q^{(n)})^{x_{k-1}}_{x_k}(Q^{(n)})^{x_k}_{x_1}.
\end{equation}
A loop is an equivalence class of pointed loops: two pointed loops are equivalent iff they are the same under a rotation, e.g. the pointed loop $(1,2,3,4)$ is equivalent to $(2,3,4,1)$, but not $(1,2,4,3)$. We denote by $\ell$ the equivalence class of the pointed loop $\dot{\ell}$. The loop measure $\mu_n$ is the corresponding push-forward measure of the pointed loop measure $\dot{\mu}_n$ on the space of loops.

Denote by $\mathcal{DL}_{\alpha}^{(n)}$ the Poisson point process (or ``loop-soup'') of non-trivial loops with intensity measure $\alpha\mu_n$ where $\alpha>0$ is a fixed parameter. We view it as a multiset, by identifying it with its support. For example, if $\mathcal{DL}_{\alpha}^{(n)}=3\delta_{\ell_1}+2\delta_{\ell_2}$, we write $\mathcal{DL}_{\alpha}^{(n)}=\{\ell_1,\ell_1,\ell_1,\ell_2,\ell_2\}$. As in \cite{Markovian-loop-clusters-on-graphs}, we define the loop clusters as follows: 
\begin{defn}[Loop clusters]\label{defn: loop clusters}
 Given a realization of the loop-soup $\mathcal{DL}$, an undirected edge $\{x,y\}$ is closed iff there is no loop in the loop-soup $\mathcal{DL}$ which covers $\{x,y\}$ in any direction. Otherwise, we say that the undirected edge $\{x,y\}$ is open. For two vertices $x$ and $y$, we say that $x$ is connected to $y$ by the loop-soup $\mathcal{DL}$ if either $x=y$ or $x$ and $y$ are connected through open edges, which is denoted by $x\overset{\mathcal{DL}}{\longleftrightarrow}y$. Note that $\overset{\mathcal{DL}}{\longleftrightarrow}$ is an equivalence relation, which naturally defines a partition $\Pi(\mathcal{DL},G)$ of the vertex set $V$ of the graph $G$. Each partition is called a loop cluster. For simplicity of notation, we denote by $\mathcal{C}^{(n)}_{\alpha}$ the partition $\Pi(\mathcal{DL}_{\alpha}^{(n)},G^{(n)})$ associated with $\mathcal{DL}_{\alpha}^{(n)}$ on the discrete circle $G^{(n)}$.
\end{defn}

Note that our loop-soup is slightly different from the loop-soup considered by Lemaire, Le Jan and A.-S. Sznitman as we exclude trivial loops\footnote{A (pointed) loop $(x)$ of a single vertex is called trivial.}. The reason is that trivial loops contribute nothing to the loop clusters according to our definition of open edges. \emph{Therefore, we only consider the non-trivial loops in this paper. Sometimes, we omit the word ``non-trivial'' for the simplicity of notation.}

\subsection{Known results on a subinterval of \texorpdfstring{$\mathbb{Z}$}{Z}}
\begin{defn}\label{defn: sym model on Z}
 By Model$(I,\alpha,\kappa)$, we mean the Poisson point process $\mathcal{DL}_{\alpha}$ of loops on discrete interval $I\subset\mathbb{Z}$ of intensity $\alpha\mu^{(\kappa)}$, where $\alpha,\kappa\geq 0$ are two parameters and the loop measure $\mu^{(\kappa)}$ is the push forward measure of the following pointed loop measure:
\begin{equation}
 \dot{\mu}^{(\kappa)}(\dot{\ell}=(x_1,\ldots,x_m))=\left\{
 \begin{array}{ll}
  \frac{1}{m}\left(\frac{1}{1+\kappa/2}\right)^m & \text{if }\dot{\ell}\text{ is a nearest-neighbor loop on }I,\\
  0 & \text{otherwise.}                                                        
 \end{array}\right.
\end{equation}
\end{defn}

In \cite[Section 5]{Amas-de-lacets-markoviens}, Le Jan studied Model$([1,N],\alpha,0)$ and obtained the following description of the loop clusters and represented the scaling limit by a stable subordinator. Also, Le Jan pointed out the relation among models on different discrete intervals.
\begin{thm}\label{thm: lejan}\cite[Section 5]{Amas-de-lacets-markoviens}
Consider the interval $I=[1,N]$ and $\kappa=0$.
\begin{itemize}
 \item $N=\infty$: for $\alpha\in]0,1[$, the left end points of the closed edges at time $\alpha$ form a renewal process with holding times $(W_i)_{i\geq 1}$; The generating function of $W_1$ is $1-\frac{s}{\Li_{\alpha}(s)}$ where $\Li_{\alpha}$ denotes the polylogarithm: $\forall |s|<1$, $\Li_{\alpha}(s)=\sum\limits_{k=1}^{+\infty}\frac{s^k}{k^{\alpha}}$. Set $S_n=\sum\limits_{i=1}^{n}W_{i}$ for $n\geq 1$. As $\epsilon$ tends to $0$, $(\epsilon S_{\lfloor\epsilon^{\alpha-1}t\rfloor},t\geq 0)$ converges in law towards a stable subordinator with index $1-\alpha$. For $\alpha>1$, there are only a finite number of clusters. In particular, $\mathbb{P}[S_1=\infty]=\frac{1}{\zeta(\alpha)}$.
 \item $N<\infty$: we obtain a renewal process conditioned to jump at point $N$.
\end{itemize}
\end{thm}

Later, in \cite{Markovian-loop-clusters-on-graphs}, as an example, Lemaire and Le Jan studied the Model$(\mathbb{Z},\alpha,\kappa)$. Their result describes the law of the closed edges and the scaling limit:
\begin{thm}\label{thm: lemaire lejan}\cite[Proposition 3.1]{Markovian-loop-clusters-on-graphs}
 Set $r(\kappa)=\log\left(1+\frac{\kappa}{2}+\sqrt{\kappa+\frac{\kappa^2}{4}}\right)$.
 \begin{itemize}
  \item The midpoints of the closed edges form a renewal process. Moreover, for $n\in\mathbb{Z}$,
  \begin{equation*}
   \mathbb{P}[\{n,n+1\}\text{ is closed }]=(1-e^{-2r(\kappa)})^{\alpha},
  \end{equation*}
  \begin{equation}\label{eq: tll1}
  \mathbb{P}[\{n,n+1\}\text{ is closed }|\{0,1\}\text{ is closed }]=\frac{(1-e^{-2 r(\kappa)})^{\alpha}}{(1-e^{-2(n+1)r(\kappa)})^{\alpha}}.
  \end{equation}
  \item Assume that $\alpha\in]0,1[$. Denote by $\nu^{(\kappa)}$ the law of this renewal process, that is, the law of the distance between the left end points of two consecutive closed edges. For $\epsilon>0$, denote by $(W_{\epsilon,i})_{i\in\mathbb{N}_{+}}$ a sequence of independent random variables with distribution $\nu^{(\epsilon\kappa)}$. For every $t>0$, as $\epsilon\rightarrow 0$, the variable $\sqrt{\epsilon}\sum\limits_{i=1}^{[\epsilon^{-(1-\alpha)/2}t]}W_{\epsilon,i}$ converges in law to the value at $t$ of a subordinator with potential density $U(x,y)=\left(\frac{2\sqrt{\kappa}}{1-e^{-2|x-y|\sqrt{\kappa}}}\right)^{\alpha}$.
 \end{itemize}
\end{thm}
\begin{rem}\label{rem: lemairelejan skorokhod}
 Although the convergence is stated for a fixed time $t$, we actually have the convergence in distribution of the finite marginals by Markov property. Moreover, by strong Markov property, they satisfy Aldous' criteria for the tightness, and so, the result could be strengthened to the convergence in Skorokhod space, see Lemma \ref{lem: tightness and skorokhod convergence towards a subordinator}.
\end{rem}

\subsection{Presentation of our results}

In this article, we consider the loop clusters in the discrete circle $G^{(n)}$. We fix some notation which will be frequently used in the sequel.
\begin{defn}\label{defn: notation}
 Set $\kappa^{(n)}\overset{\mathrm{def}}{=}\frac{1+c_n-2\sqrt{p_n(1-p_n)}}{\sqrt{p_n(1-p_n)}}$ and $r^{(n)}=\log\left(1+\frac{\kappa^{(n)}}{2}+\sqrt{\kappa^{(n)}+\frac{(\kappa^{(n)})^2}{4}}\right)$.
\end{defn}
\begin{defn}
 For a loop-soup $\mathcal{DL}$ (i.e. a Poisson point process of loops), we view it as a multiset. We write $\mathcal{DL}=\emptyset$ iff $\mathcal{DL}=0$ as a random point measure. For two loop-soups $\mathcal{DL}$ and $\mathcal{DL}'$, we write $\mathcal{DL}\cup\mathcal{DL}'$ instead of $\mathcal{DL}+\mathcal{DL}'$.
\end{defn}

We write $\mathcal{DL}_{\alpha}$ as sums of four independent Poisson point process $(\mathcal{DL}^{(n)}_{\alpha,i})_{i=1,2,3,4}$ of loops, which will be specified later in Definition \ref{defn: four type of loops}. For the present, we would like to mention that
\begin{itemize}                                                                                                                                                                                                                                                         \item $\mathcal{DL}^{(n)}_{\alpha,1}$ is $\mathcal{DL}_{\alpha}^{(n)}$ restricted on the loops avoiding the vertex $1$,
\item $\mathcal{DL}^{(n)}_{\alpha,2}\cup\mathcal{DL}^{(n)}_{\alpha,3}\cup\mathcal{DL}^{(n)}_{\alpha,4}$ are loops passing through the vertex $1$.                                                                                                                                                                                                                                                                 \end{itemize}

Our argument contains three steps:
\begin{itemize}
 \item We study the loop clusters conditionally on that $\mathcal{DL}^{(n)}_{\alpha}=\mathcal{DL}^{(n)}_{\alpha,1}$. We will use Theorem \ref{thm: lemaire lejan} (\cite[Proposition 3.1]{Markovian-loop-clusters-on-graphs}) as our starting point.
 \item We study the loop clusters conditionally on that $\mathcal{DL}^{(n)}_{\alpha,1}=\emptyset$, which does not appear in the loop cluster model on discrete intervals.
 \item By combining the results in the previous two steps, we get a full description of the loop clusters.
\end{itemize}

We start with the first step: our first observation is the following description of the loop clusters given that $\mathcal{DL}^{(n)}_{\alpha}=\mathcal{DL}^{(n)}_{\alpha,1}$ (or equivalently, $\mathcal{DL}^{(n)}_{\alpha,2}\cup\mathcal{DL}^{(n)}_{\alpha,3}\cup\mathcal{DL}^{(n)}_{\alpha,4}=\emptyset$). 
\begin{prop}\label{conditioned renewal process}
Conditionally on $\mathcal{DL}^{(n)}_{\alpha,2}\cup\mathcal{DL}^{(n)}_{\alpha,3}\cup\mathcal{DL}^{(n)}_{\alpha,4}=\emptyset$,
\begin{itemize}
 \item[a)] our model is the same as Model$([2,n],\alpha,\kappa^{(n)})$ where $\kappa^{(n)}$ is given by Definition \ref{defn: notation},
 \item[b)] \cite[Proposition 3.1]{Markovian-loop-clusters-on-graphs}, the left points of these closed edges, together with the left points of $\{0,1\}$ and $\{n,n+1\}$, form a renewal process $(S^{(\kappa^{n})}_i)_{i\geq 0}$ $(S^{(\kappa^{n})}_0=0)$ conditioned to hit $n$, where the generating function\footnote{This generating function has already been known by Lemaire and Le Jan, see the proof of \cite[Proposition 3.1]{Markovian-loop-clusters-on-graphs}.} $\Psi^{(\kappa^{(n)})}(s)=\mathbb{P}\left[s^{S^{(\kappa^{(n)})}_1-S^{(\kappa^{(n)})}_0}\right]$ of the jump distribution $S^{(\kappa^{(n)})}_1-S^{(\kappa^{(n)})}_0$ is given by
 \begin{equation}\label{eq: crp1}
  (1-\Psi^{(\kappa^{(n)})}(s))^{-1}=\sum\limits_{n\geq 0}\left(\frac{1-\exp\{-2 r^{(n)}\}}{1-\exp\{-(n+1)r^{(n)}\}}\right)^{\alpha}s^n,
 \end{equation}
 where $\kappa^{(n)}$ and $r^{(n)}$ are given in Definition \ref{defn: notation}.
\end{itemize}
\end{prop}

It is natural to believe that the renewal processes conditioned to hit $n$, rescaled by $1/n$, as $n\rightarrow\infty$, converges to a subordinator conditioned to hit $1$. Indeed, we will prove this in the following proposition. 

\begin{prop}\label{convergence of conditioned renewal processes}
 Assume that $\alpha\in]0,1[$ and that $\lim\limits_{n\rightarrow\infty}n^2\kappa^{(n)}=\kappa$. Then, let $\left(W^{(\kappa^{(n)})}_i\right)_{i\geq 1}$ be a sequence of i.i.d. variables with the generator function $\Phi^{(\kappa^{(n)})}$ defined in Proposition \ref{conditioned renewal process}. For $m\geq 0$, let $S^{(\kappa^{(n)})}_m$ be the partial sum of $\left(W^{(\kappa^{(n)})}_i\right)_{i\geq 1}$, i.e. $S^{(\kappa^{(n)})}_m\overset{\mathrm{def}}{=}\sum\limits_{i=1}^{m}W^{(\kappa^{(n)})}_i$. Set $T_{]1,+\infty[}^{(n)}=\inf\{t\geq 0:\frac{1}{n}S^{(\kappa^{(n)})}_{\lfloor n^{1-\alpha}t\rfloor}>1\}$. Let $(X^{(\kappa)}_t)_{t\geq 0}$ be the subordinator of the potential density $U(x,y)=\left(\frac{2\sqrt{\kappa}}{1-e^{-2|x-y|\sqrt{\kappa}}}\right)^{\alpha}$ and $T_{]1,+\infty[}=\inf\{t\geq 0:X^{(\kappa)}>1\}$. Then, we have the following convergence result in the Skorokhod space.
 \begin{multline*}
  \lim\limits_{n\rightarrow\infty}\mathbb{P}\left[\left.\left(\frac{1}{n}S^{(\kappa^{(n)})}_{\lfloor n^{1-\alpha} t\rfloor}\right)_{t\in\left[0,T^{(n)}_{]1,+\infty[}\right[}\in\cdot\right|S^{(\kappa^{(n)})}\text{ hits }n\right]\\
  =\mathbb{P}\left[\left.(X_t)_{t\in\left[0,T_{]1,+\infty[}\right[}\in \cdot\right|X_{T_{]1,+\infty[}-}=1\right].
 \end{multline*}
\end{prop}

The conditioned subordinator is a well-defined Feller process, see Lemma \ref{lem: subordinator bridge}. Our result, Proposition \ref{conditioned renewal process} together with Proposition \ref{convergence of conditioned renewal processes}, is a conditioned version of Theorem \ref{thm: lemaire lejan} (\cite[Proposition 3.1]{Markovian-loop-clusters-on-graphs}). The convergence of conditioned renewal processes is not included in Theorem \ref{thm: lemaire lejan}  (\cite[Proposition 3.1]{Markovian-loop-clusters-on-graphs}). Some additional argument is necessary, see Subsection \ref{subsect: proof of convergence of conditioned renewal processes}. Also, note that Theorem \ref{thm: lemaire lejan} (\cite[Proposition 3.1]{Markovian-loop-clusters-on-graphs}) is stated for a fixed time $t$. Here, we state the convergence result in Skorokhod space. The reason is that the finite marginal convergence does not imply a convergence result for general clusters. (For instance, if we split the biggest cluster into two clusters by adding a closed edge in the middle of that cluster, then we still have the same limit for finite marginals with a different limit for the clusters.) On the other hand, the Skorokhod convergence in Proposition \ref{convergence of conditioned renewal processes} does imply the convergence of the macroscopic jumps. In other words, it implies the convergence of the loop clusters in the following sense:

\begin{cor}\label{cor: cvg of clusters interval}
 Let $\bar{\mathcal{R}}$ be the closure of the range of $(Y^{(\kappa)}_t)_{t\in[0,\zeta[}$ where
 $$\mathbb{P}[(Y^{(\kappa)}_t)_{t\in[0,\zeta[}\in\cdot]\overset{\mathrm{def}}{=}\mathbb{P}[(X_t)_{t\in[0,T_{]1,+\infty[}-[}\in \cdot|X_{T_{]1,+\infty[}-}=1].$$
 Then, its complementary consists of countably many open intervals. We list them in the decreasing order according to the lengths: $(g_1,d_1),(g_2,d_2),\ldots$ ($g_1-d_1\geq g_2-d_2 \geq \cdots$). Similarly, conditionally on $\mathcal{DL}^{(n)}_{\alpha,2}\cup\mathcal{DL}^{(n)}_{\alpha,3}\cup\mathcal{DL}^{(n)}_{\alpha,4}=\emptyset$, the discrete circle $G^{(n)}$ is divided into several discrete arcs\footnote{It is possible that some discrete arcs are actually single vertices. For example, if $\{1,2\}$ and $\{2,3\}$ are both closed, then the vertex $2$ is considered to be a discrete arc $[2,2]$.} by closed edges. We list them in decreasing order according to the lengths: $[g^{(n)}_1,d^{(n)}_1],[g^{(n)}_2,d^{(n)}_2],\ldots,[g^{(n)}_{k_n},d^{(n)}_{k_n}]$ where $k_n=\#\mathcal{C}^{(n)}_{\alpha}$ is the number of the discrete arcs. Assume that $\alpha\in]0,1[$ and that $\lim\limits_{n\rightarrow\infty}n^2\kappa^{(n)}=\kappa$. Then, for each $k\geq 1$, conditionally on $\mathcal{DL}^{(n)}_{\alpha,2}\cup\mathcal{DL}^{(n)}_{\alpha,3}\cup\mathcal{DL}^{(n)}_{\alpha,4}=\emptyset$, the random variable $\frac{k_n}{n^{1-\alpha}}$ converges in distribution to the time duration $\zeta$ of the process $Y^{(\kappa)}$ and
 $$\frac{1}{n}(g^{(n)}_1,d^{(n)}_1,g^{(n)}_2,d^{(n)}_2,\ldots,g^{(n)}_k,d^{(n)}_k)\text{ converges in distribution to }(g_1,d_1,g_2,d_2,\ldots,g_k,d_k).$$
 Equivalently, as $n\rightarrow\infty$, the compact set $\frac{1}{n}\left([1,n]\setminus\bigcup\limits_{i}]g^{(n)}_i,d^{(n)}_i[\right)$ converges in law to $\bar{\mathcal{R}}$, with respect to the Hausdorff distance between compact sets.
\end{cor}
 We give the density of the time duration $\zeta$ of $Y^{(\kappa)}$ (or the limit distribution of $n^{\alpha-1}k_n$), by using the density of the semi-group of the subordinator $X^{(\kappa)}$ in the following remark.
\begin{rem}
 Denote by $P^{(\kappa)}_t(x,\mathrm{d}y)$ the semi-group of the subordinator $(X^{(\kappa)}_t)_{t\geq 0}$ with potential density $U(x,y)=\left(\frac{2\sqrt{\kappa}}{1-e^{-2|x-y|\sqrt{\kappa}}}\right)^{\alpha}$. By Fourier analysis, we can show that $P_t^{(\kappa)}(x,\mathrm{d}y)$ has a density $p_t^{(\kappa)}(x,y)$ with respect to the Lebesgue measure. Moreover, $p_t^{(\kappa)}(x,y)$ is jointly continuous in $(t,x,y)$. Later, we will see in Lemma \ref{lem: subordinator bridge} that $Y^{(\kappa)}$ is a Doob's harmonic transform of $X^{(\kappa)}$ and that the semi-group $Q^{(\kappa)}_t(x,\mathrm{d}y)$ of $Y^{(\kappa)}$ has the following form,
 \begin{equation}
  Q^{(\kappa)}_t(x,\mathrm{d}y)=\frac{U(y,1)}{U(x,1)}P^{(\kappa)}_t(x,\mathrm{d}y).
 \end{equation}
 Immediately, we see that $Q^{(\kappa)}_t(x,\mathrm{d}y)$ has a density $q_t^{(\kappa)}(x,y)=\frac{U(y,1)}{U(x,1)}p_t^{(\kappa)}(x,y)$ with respect to the Lebesgue measure such that $(t,x,y)\rightarrow q_t^{(\kappa)}(x,y)$ is jointly continuous. By semi-group property,
 $$\mathbb{P}[\zeta>t]=\int Q^{(\kappa)}_t(0,\mathrm{d}y)=\int\frac{U(y,1)}{U(0,1)}p_t^{(\kappa)}(0,y)=\frac{1}{U(0,1)}\int\limits_{t}^{\infty}p_s^{(\kappa)}(0,1)\,\mathrm{d}s.$$
 Thus, the density of $\zeta$ with respect to the Lebesgue measure is exactly $\frac{p_t^{(\kappa)}(0,1)}{U(0,1)}$.
\end{rem}
\bigskip
Next, we study the loops passing through the vertex $1$ conditionally on that $\mathcal{DL}^{(n)}_{\alpha,1}=\emptyset$. We need some notation to represent the cluster formed by the loops passing through the vertex $1$.
\begin{defn}
 If $\mathcal{DL}^{(n)}_{\alpha,1}=\emptyset$ and if there exist at least two clusters, then there exist two end points of the discrete loop cluster containing $1$. We denote by $J_n$ the graph distance, inside the loop cluster containing $1$, between $1$ and the left end point, and by $K_n$ the distance between $1$ and the right end point. Necessarily, $J_n+K_n\leq n-2$.
\end{defn}
For example, if the left end point is $n-1$ and the right end point is $4$, then $J_n=2$ and $K_n=3$. We give an explicit description of the loop clusters conditionally on that $\mathcal{DL}^{(n)}_{\alpha,1}=\emptyset$:
\begin{prop}\label{loops through 1}
 \begin{align*}
  \mathbb{P}[\exists \geq 2\text{ loop clusters}|\mathcal{DL}^{(n)}_{\alpha,1}=\emptyset]=&2^{\alpha}\left(\frac{\cosh (nr^{(n)})-\cosh\left(\frac{1}{2}n\log\left(\frac{p_n}{1-p_n}\right)\right)}{\sinh (nr^{(n)})}\right)^{\alpha}\\
  &\times\left(\sum\limits_{m=1}^{n-1}\left(\frac{\sinh (mr^{(n)})}{\sinh (nr^{(n)})}\right)^{\alpha}-\left(\frac{\sinh ((m-1)r^{(n)})}{\sinh ((n-1)r^{(n)})}\right)^{\alpha}\right).
 \end{align*}
 For two non-negative integers $m$ and $M$ such that $m+M\leq n-2$, we have that
 \begin{multline*}
  \mathbb{P}[\exists\geq 2\text{ loop clusters}, J_n\leq m,K_n\leq M|\mathcal{DL}^{(n)}_{\alpha,1}=\emptyset]\\=2^{\alpha}\left(\frac{\cosh (nr^{(n)})-\cosh\left(\frac{1}{2}n\log\left(\frac{p_n}{1-p_n}\right)\right)}{\sinh (nr^{(n)})}\right)^{\alpha}\\
  \times\left(\frac{\sinh ((m+1)r^{(n)})\sinh ((M+1)r^{(n)})}{\sinh ((m+M+2)r^{(n)})}\right)^{\alpha}.
 \end{multline*}

 Suppose that $\lim\limits_{n\rightarrow\infty}n^2\kappa^{(n)}=\kappa\geq 0$ and that $\lim\limits_{n\rightarrow\infty}n^2c_n=\epsilon\in[0,\kappa/2]$. Then,
 \begin{multline*}
  \lim\limits_{n\rightarrow\infty}\mathbb{P}[\exists \geq 2\text{ loop clusters}|\mathcal{DL}^{(n)}_{\alpha,1}=\emptyset]\\
  =2^{\alpha}\cdot\alpha\sqrt{\kappa}\frac{(\cosh(\sqrt{\kappa})-\cosh(\sqrt{\kappa-2\epsilon}))^{\alpha}}{(\sinh\sqrt{\kappa})^{2\alpha+1}}\int\limits_{0}^{1}\left(\sinh(a\sqrt{\kappa})\right)^{\alpha-1}\left(\sinh((1-a)\sqrt{\kappa})\right)^{\alpha+1}\,\mathrm{d}a.
 \end{multline*}
 For $a,b\geq 0$ such that $a+b\leq 1$, we have that
 \begin{multline*}
  \mathbb{P}\left[\left.\exists\geq 2\text{ loop clusters}, \frac{J_n}{n}\leq a,\frac{K_n}{n}\leq b\right|\mathcal{DL}^{(n)}_{\alpha,1}=\emptyset\right]\\
  =2^{\alpha}\left(\frac{\cosh\sqrt{\kappa}-\cosh\sqrt{\kappa-2\epsilon}}{\sinh\sqrt{\kappa}}\right)^{\alpha}\left(\frac{\sinh (a\sqrt{\kappa})\sinh (b\sqrt{\kappa})}{\sinh ((a+b)\sqrt{\kappa})}\right)^{\alpha}.
 \end{multline*}
\end{prop}

We have obtained the description of the partition of loop clusters $\Pi(\mathcal{DL}_{\alpha,1}^{(n)},G^{(n)})$ formed by the Loop-soup $\mathcal{DL}_{\alpha,1}^{(n)}$ avoiding $1$, and the loop cluster $CL^{(n)}$ by the loop-soup $\mathcal{DL}_{\alpha}^{(n)}\setminus\mathcal{DL}_{\alpha,1}^{(n)}$ of loops intersecting $1$. Then, the loop clusters $\Pi(\mathcal{DL}_{\alpha}^{(n)},G^{(n)})$ formed by the loop-soup $\mathcal{DL}_{\alpha}^{(n)}$ is determined as follows:
$$\Pi(\mathcal{DL}_{\alpha}^{(n)},G^{(n)})=\left\{CL^{(n)}\cup\bigcup\limits_{\substack{P\in \Pi(\mathcal{DL}_{\alpha,1}^{(n)},G^{(n)})\\P\cap CL^{(n)}\neq\emptyset}}P\right\}\cup\{P: P\in \Pi(\mathcal{DL}_{\alpha,1}^{(n)},G^{(n)}), P\cap CL^{(n)}=\emptyset\}.$$
We give the scaling limit in the following theorem.

\begin{thm}\label{thm: limit distribution of loop cluster}
 Suppose that $\lim\limits_{n\rightarrow\infty}n^2\kappa^{(n)}=\kappa$ where $\kappa^{(n)}\overset{\mathrm{def}}{=}\frac{1+c_n-2\sqrt{p_n(1-p_n)}}{\sqrt{p_n(1-p_n)}}$. Suppose that $\lim\limits_{n\rightarrow\infty}n^2c_n=\epsilon\in[0,\kappa/2]$. Let $\mathcal{C}^{(n)}_{\alpha}$ be the partition given by loop clusters on discrete circle which is defined in the introduction. If $\mathcal{C}^{(n)}_{\alpha}$ is not a single partition, then there exist $G_n\geq 0$ and $D_n\geq 0$ such that $D_n+G_n<n$ and that 
$$\{-G_n+n+1,\ldots,n,1,\ldots,1+D_n\}$$
is the cluster containing $1$. In this case, let $$1+D_n=S_0^{(n)}<S_1^{(n)}<\cdots<S_{k(n)}^{(n)}=n-G_n$$
be all the left end points of the closed edges. Define the scaled process by
$$\tilde{S}^{(n)}_t=\frac{1}{n-1-G_n-D_n}(S^{(n)}_{\lfloor (n-1-G_n-D_n)^{1-\alpha}t\rfloor}-S^{(n)}_0).$$
Let $(G,D)$ be a pair of variables with the following density
$$1_{\{x,y>0,x+y<1\}}\frac{\sin(\alpha\pi)}{\pi}\frac{2^{\alpha-2}(1-\alpha)\kappa\sinh\sqrt{\kappa}}{\sinh(\sqrt{\kappa}(1-\alpha))\left[\sinh(\sqrt{\kappa}(1-x-y))\right]^{\alpha}\left[\sinh(\sqrt{\kappa}(x+y))\right]^{2-\alpha}}.$$
Let $Y^{(\kappa)}$ be a conditioned subordinator described in Lemma \ref{lem: subordinator bridge}.  
\begin{itemize}
 \item[a)] For $\alpha\geq 1$, we have that $\lim\limits_{n\rightarrow\infty}\mathbb{P}[\mathcal{C}^{(n)}_{\alpha}\text{ is a single partition}]=1$. For $\alpha\in]0,1[$, we have that
\begin{multline*}
\lim\limits_{n\rightarrow\infty}\mathbb{P}[\mathcal{C}^{(n)}_{\alpha}\text{ is not a single partition}]\\
=\frac{1}{\sinh\sqrt{\kappa}}2^{\alpha}\sinh(\sqrt{\kappa}(1-\alpha))(\cosh(\sqrt{\kappa})-\cosh(\sqrt{\kappa-2\epsilon}))^{\alpha}.
\end{multline*}
 \item[b)] Fix $\alpha\in]0,1[$. Conditionally on that $\mathcal{C}^{(n)}_{\alpha}$ is not a single partition, $(\frac{G_n}{n},\frac{D_n}{n},\tilde{S}^{(n)})$ converges in distribution to $(G,D,M)$. Conditionally on $(G,D)$, the process $M$ has the same distribution as $Y^{(\kappa(1-G-D)^{2})}$. In particular, similar to Corollary \ref{cor: cvg of clusters interval}, this implies the convergence of $\Pi(\mathcal{DL}_{\alpha}^{(n)},G^{(n)})$. To reduce the amount of notation, we state the convergence result in an equivalent way by considering the closed edges: Let $\mathcal{S}^{(n)}=\{S^{(n)}_0,\ldots,S^{(n)}_{k(n)}\}$ be the set of the left end points of the closed edges on $G^{(n)}$. Then, $\frac{1}{n}\mathcal{S}^{(n)}\subset [0,1]$. Let $\bar{\mathcal{R}}(M)$ be the closure of the range of the process $M$, then $G+(1-G-D)\bar{\mathcal{R}}(M)$ is a compact subset of $[0,1]$. We equip the space $\mathcal{K}[0,1]$ of compact subsets of $[0,1]$ with the Hausdorff metric. Then, as $n\rightarrow\infty$, conditionally on that $\mathcal{C}^{(n)}_{\alpha}$ is not a single partition, $\frac{1}{n}\mathcal{S}^{(n)}$ converges in law to $G+(1-G-D)\bar{\mathcal{R}}(M)$.
\end{itemize}
\end{thm}

\subsection{Connection with known results, difficulties and techniques}\label{subsect: connection}

Let $\sigma^{(n)}$ be an independent uniform random permutation of $\{1,\ldots,n\}$. Denote by $\sigma^{(n)}(\mathcal{C}_{\alpha}^{(n)})$ the permuted partition of $\{1,\ldots,n\}$ such that two vertices $x,y$ belong to the same cluster of the partition $\sigma^{(n)}(\mathcal{C}_{\alpha}^{(n)})$ iff $\sigma^{-1}(x),\sigma^{-1}(y)$ belong to the same cluster of $\mathcal{C}_{\alpha}^{(n)}$. Conditionally on $\mathcal{DL}_{\alpha}^{(n)}=\mathcal{DL}_{\alpha,1}^{(n)}$ (i.e. no loop passes through $1$), $\sigma^{(n)}(\mathcal{C}_{\alpha}^{(n)})$ is a Gibbs partition\footnote{See \cite[Equations (1.47) (1.48), Section 1.5]{PitmanMR2245368} for a precise definition.}: for each particular partition $\{A_1,\ldots,A_k\}$ of $\{1,\ldots,n\}$,
 \begin{equation}
  \mathbb{P}\left[\sigma^{(n)}(\mathcal{C}_{\alpha}^{(n)})=\{A_1,\ldots,A_k\}\Big|\mathcal{DL}_{\alpha}^{(n)}=\mathcal{DL}_{\alpha,1}^{(n)}\right]=\frac{v_k\prod\limits_{i=1}^{k}\left(w_{\# A_i}\# A_i!\right)}{B_n}
 \end{equation}
 where $v_k=k!$ for $k=1,2,3,\ldots$, $w=(w_i)_i$ is the jumping distribution of the renewal process given by Equation \eqref{eq: crp1} and $B_n$ is the normalizing constant such that $B_n/n!$ equals to Equation \eqref{eq: tll1} with $\kappa=\kappa^{(n)}=\frac{1+c_n-2\sqrt{p_n(1-p_n)}}{\sqrt{p_n(1-p_n)}}$.

 For a consistent family of Gibbs partitions, (more generally, for a consistent family of exchangeable partitions), one has the almost surely convergence of the normalized sizes of equivalence classes, which is known as Kingman's representation theorem, see for example \cite[Theorem 2.2]{PitmanMR2245368}. For $\kappa=0$, conditionally on $\mathcal{DL}_{\alpha}^{(n)}=\mathcal{DL}_{\alpha,1}^{(n)}$, our family of permuted partitions form a consistent family of Gibbs partition which is driven by a mixture of $1-\alpha$ stable subordinator bridge, see \cite[Theorem 4.6]{PitmanMR2245368}. However, for $\kappa\neq 0$, conditionally on $\mathcal{DL}_{\alpha}^{(n)}=\mathcal{DL}_{\alpha,1}^{(n)}$, our family of permuted partitions $\sigma^{(n)}(\mathcal{C}_{\alpha}^{(n)})$ is not consistent. (One can argue this by using \cite[Theorem 4.6]{PitmanMR2245368}.)

 There exists convergence results of Gibbs partitions for non-consistent family of Gibbs partitions, see for example \cite[Theorem 2.4,Theorem 2.5]{PitmanMR2245368} with the references. However, they put an assumption that the sequence $w$ does not depend on $n$. From our point of view, they put this condition to get a convergence towards a subordinator bridge by applying a local limit theorem of I. A. Ibragimov and Y. V. Linnik \cite[Chapter 4]{IbragimovLinnikMR0322926}. That local limit theorem, stated for distributions in the attraction domain of some stable distribution, is not applicable in our situation (as our limit distribution is not stable, see Proposition \ref{convergence of conditioned renewal processes}). Rather than establishing a local limit theorem for our case, we prove the convergence of the conditioned renewal processes by the convergence of renewal process in \cite[Proposition 3.1]{Markovian-loop-clusters-on-graphs}. Also, we would like to mention a general result of O. Kallenberg \cite[Theorem 16.23]{KallenbergMR1876169}\footnote{The result first appeared in his paper \cite{KallenbergMR0394842}.} on the convergence of discrete exchangeable processes towards an exchangeable process on $[0,1]$. Kallenberg formulated an equivalence condition for convergence in Skorokhod space $D[0,1]$. However, for our loop model, roughly speaking, that condition requires the convergence of macroscopic clusters which needs to be proven.
 
 For $\kappa>0$, our limit partition is driven by a subordinator different from the stable $1-\alpha$ subordinator, which appears in \cite[Theorem 2.5]{PitmanMR2245368} (and corresponds to the case $\kappa=0$). However, our scaling function $n^{1-\alpha}$ has the same form as in \cite[Theorem 2.5]{PitmanMR2245368}. More precisely, under the assumption that $\lim\limits_{n\rightarrow\infty}n^2\kappa^{(n)}=\kappa\in[0,\infty[$, the scaling limit of the partition depends on the limit $\kappa$, but the scaling function $n^{1-\alpha}$ doesn't. We would like to briefly explain the reason as follows: suppose we have two sequences of models $(Model_n)_n$ and $(\widetilde{Model}_n)_n$ on $(G^{(n)})_n$ with parameters $(\kappa^{(n)})_n$ and $(\tilde{\kappa}^{(n)})_n$ respectively. Let us suppose further that                                                                                                                                                                                                                                                                                                                                                                                                                                                                                                                                                                                                                                                                                                                                                                                                                                                                                                                                                            $$\forall n\geq 1, \kappa^{(n)}<\tilde{\kappa}^{(n)}\text{ and } \lim\limits_{n\rightarrow\infty}n^2\kappa^{(n)}<\lim\limits_{n\rightarrow\infty}n^2\tilde{\kappa}^{(n)}.$$
 From the construction of loop-soup as a Poisson point process, the loop-soup $\mathcal{DL}_{\alpha}^{(n)}$ in the model $Model_n$ can be constructed from the loop-soup $\widetilde{\mathcal{DL}}_{\alpha}^{(n)}$ in the model $\widetilde{Model}_n$ by adding an additional independent Poisson point process. The intensity measure is equal to the difference $\alpha\mu^{(n)}-\alpha\tilde{\mu}^{(n)}$, where $\mu^{(n)}$ and $\tilde{\mu}^{(n)}$ are intensity measures of $\mathcal{DL}_{\alpha}^{(n)}$ and $\widetilde{\mathcal{DL}}_{\alpha}^{(n)}$ respectively. Under our assumption, $\alpha\mu^{(n)}-\alpha\tilde{\mu}^{(n)}$ are uniformly bounded for all $n$. Thus, $\#(\widetilde{\mathcal{DL}}_{\alpha}^{(n)}\setminus\mathcal{DL}_{\alpha}^{(n)})$ is a Poisson random variable with uniformly bounded expectation. Moreover, with high probability, the loops inside $\widetilde{\mathcal{DL}}_{\alpha}^{(n)}\setminus\mathcal{DL}_{\alpha}^{(n)}$ are macroscopic loops away from $1$. Consequently, $\inf\limits_{n\geq 1}\mathbb{P}[Model_n=\widetilde{Model}_n]>0$. Also, we have the same scaling function for different possible limit $\lim\limits_{n\rightarrow\infty}n^2\kappa^{(n)}$. The same idea shows that the convergence result remains the same if we perturb the killing parameter $\kappa^{(n)}$ (by changing $p_n,c_n$) up to order $o(n^{-2})$. 
 
\bigskip

 Finally, we briefly present the difficulties and the techniques. To prove Proposition \ref{convergence of conditioned renewal processes}, we would like to use the convergence result \cite[Proposition 3.1]{Markovian-loop-clusters-on-graphs} of the renewal processes. The conditioned subordinator in Proposition \ref{convergence of conditioned renewal processes} is well-defined by Doob's $h$-transform in Lemma \ref{lem: subordinator bridge}. The difficulty of proving the convergence is due to the divergence of the Radon-Nikodym derivatives between the conditioned renewal processes and the renewal processes. However, we have the convergence of the conditional expectations of the Radon-Nikodym derivatives on some sub-$\sigma$-fields. As a result, we get a unique candidate for possible finite marginal limit distributions. Then, we get the convergence of finite marginal distributions. (Note that the tightness of the finite marginal distributions follows from the boundedness of the scaled processes.) To get a Skorokhod convergence, we need the tightness of the family of conditioned renewal processes. By the exchangeability (due to the connection with the conditioned renewal process in Proposition \ref{conditioned renewal process}), as an application of Aldous' criteria of tightness \cite[Theorem 16.11]{KallenbergMR1876169}, the finite marginals convergence implies the tightness, see the proof of \cite[Theorem 16.23]{KallenbergMR1876169}. Next, we consider the loops passing through the vertex $1$ which are not too large to cover the whole space. This cluster might cover some edges which are not covered by the loops avoiding $1$. Accordingly, we erase a part of the range of the conditioned subordinator which is the limit of the edges uncovered by loops avoiding $1$. Then, the remaining part of the range of the subordinator represents the closed edges in the scaling limit. For this part, the key is the independence between the loops avoiding $1$ and those loops passing through $1$ which is guaranteed by the Poisson loop-soup construction. To make it rigorous, we need the fact that the end points of the cluster formed by the loop-soup $\mathcal{DL}_{\alpha}^{(n)}\setminus\mathcal{DL}_{\alpha,1}^{(n)}$ through $1$, fall into the interior of some loop clusters formed by $\mathcal{DL}_{\alpha,1}^{(n)}$, with probability tending to $1$ as $n\rightarrow\infty$. This is guaranteed by $6$th part of Lemma \ref{lem: subordinator bridge} and the independence between $\mathcal{DL}_{\alpha}^{(n)}\setminus\mathcal{DL}_{\alpha,1}^{(n)}$ and $\mathcal{DL}_{\alpha,1}^{(n)}$. Finally, to express the results explicitly, we calculate the L\'{e}vy measure of the subordinator in Lemma \ref{lem: levy measure} by inversing Laplace transform which is unknown before this paper.

\subsection{Organization of the paper}
We would like to present the organization of the following sections:

In Section 2, we collect some useful facts on (non-trivial) loop measures by Lemaire and Le Jan, such as the restriction properties of the loop measures (Lemma \ref{lem: restriction property} and Lemma \ref{lem: several closed edges}) and the invariance under Doob's $h$-transform (Lemma \ref{lem: Doob h transform}). Also, we provide a classical result on the determinant of Toeplitz matrices (Lemma \ref{lem: Toeplitz}).

In Section 3, we prove Proposition \ref{conditioned renewal process} and Proposition \ref{convergence of conditioned renewal processes}, results of loop clusters conditioned on the absence of loops through $1$. We identify the closed edges as a renewal process conditioned to jump to $n$. Then, we give a convergence result of the conditioned renewal processes towards a conditioned subordinator.

In Section 4, we give the proof of Proposition \ref{loops through 1}, a full description of the loop clusters formed by loops through $1$, together with a limit result.

In Section 5, by combining the results in Section 3 and 4, we prove Theorem \ref{thm: limit distribution of loop cluster} of the limit distribution of the loop clusters on $G^{(n)}$ under certain conditions on the parameters.

In Section 6, we present an informal relation with Brownian loop clusters on the circle $\mathbb{S}^1$: several limit results can be predicted by Brownian loop clusters.

We postpone several proofs in the $7$-th and the last section.

\section{Useful facts}\label{sec: useful facts}
In this section, we collect some useful properties which are frequently used throughout the paper. Although we are interested in a class of special loop measures on discrete circles, we will state these properties for a general class of loop measures. For the loop measures associated with reversible Markovian chains, these results are already known by Lemaire, Le Jan, Sznitman, \ldots. These results also hold in the non-reversible case, for example, the loop-soup considered in the present paper.

Let's begin with a precise description of the loop measure. In this section, we will consider the (non-trivial) pointed loop measure associated with a discrete Markovian generator\footnote{See \cite[Definition 2.1]{yinshan} for a precise definition.} $L$ on a countable state space $S$: 
\begin{equation}\label{eq: defn gplm}
 \dot{\mu}(\dot{\ell}=(x_1,\ldots,x_k))=\frac{1}{k}Q^{x_1}_{x_2}\cdots Q^{x_k}_{x_1}\text{ for }k\geq 2, x_1,\ldots,x_k\in S 
\end{equation}
where $Q^x_y\overset{\mathrm{def}}{=}\left\{\begin{array}{ll}
-\frac{L^x_y}{L^x_x} & \text{ if }x\neq y\\
0 & \text{ if }x=y.
\end{array}\right.$
The corresponding (non-trivial) loop measure $\mu$ is the push-forward measure of $\dot{\mu}$. As we have emphasized in Subsection \ref{subsect: basic settings}, we only consider the non-trivial loops and we will omit the word ``non-trivial'' for the simplicity of notation. 

We will be interested in the loops which fulfill certain special requirements:
\begin{defn}\label{defn: inclusion exclusion vertices}[Inclusion/exclusion property, vertex set version]
Let $F$ be a subset of the state space $S$ and $\ell$ a loop on $S$. We say that $\ell$ is inside $F$ if $\ell$ does not visit any state in $S\setminus F$, denote it by $\ell\subset F$. We say that $\ell$ avoids $F$ if $\ell\subset F^c$, which is denoted by $\ell\cap F=\emptyset$. For some state $x$, we say that $\ell$ visits $x$, denoted it by $x\in\ell$, if $\ell$ doesn't avoid $\{x\}$.
\end{defn}
If we consider $S$ as a vertex set and we put directed edges between each pair $x,y\in S$, then we get a directed graph. It is natural to extend Definition \ref{defn: inclusion exclusion vertices} to an edge subset $F$.
\begin{defn}\label{defn: inclusion exclusion edges}[Inclusion/exclusion property, edge set version]
Let $F\subset S\times S$ and $\ell$ a loop on $S$. We say that $\ell=(x_1,\ldots,x_k)$ is inside $F$, which is denoted by $\ell\subset F$, if $(x_1,x_2),\ldots,(x_{k-1},x_k),(x_k,x_1)\in F$. We say that $\ell$ avoids $F$ if $\ell\subset F^c$, which is denoted by $\ell\cap F=\emptyset$.
\end{defn}

\begin{lem}\label{lem: restriction property}
 Let $\mu$ be the Markovian loop measure associated with a generator $L$ on a state space $S$. Let $F$ be a finite subset of the state space $S$. Then, $\mu(\ell\text{ is non-trivial },\ell\subset F,\mathrm{d}\ell)$ is the Markovian loop measure associated with the generator $L|_{F\times F}$. Moreover,
 $$\mu(\ell\text{ is non-trivial and }\ell\subset F)=-\log\det(-L|_{F\times F})+\sum\limits_{x\in F}\log(-L)^x_x$$
 with the convention that $-\log 0=+\infty$ and that the determinant of an empty matrix is $1$.
\end{lem}
\begin{proof}
One can deduce from \eqref{eq: defn gplm} that $\mu(\ell\subset F,\mathrm{d}\ell)$ equals the Markovian loop measure associated with the generator $L|_{F\times F}$. Hence, it remains to show that for a Markovian loop measure $\mu$ associated with the generator $L$ on a finite state space $S$,
\begin{equation}\label{eq: lrp1}
 \mu(\text{non-trivial loops})=-\log\det(-L)+\sum\limits_{x\in S}\log(-L)^x_x.
\end{equation}
By \eqref{eq: defn gplm}, we see that
\begin{equation}\label{eq: lrp2}
\mu(\text{non-trivial loops})=\dot{\mu}(\text{non-trivial pointed loops})=\sum\limits_{k\geq 2}\frac{1}{k}\Tr Q^k
\end{equation}
where
$$Q^x_y\overset{\mathrm{def}}{=}\left\{\begin{array}{ll}
-\frac{L^x_y}{L^x_x} & \text{ if }x\neq y\\
0 & \text{ if }x=y.
\end{array}\right.$$
Since $\Tr Q=0$, we have $\eqref{eq: lrp2}=\sum\limits_{k\geq 1}\frac{1}{k}\Tr Q^{k}$. It suffices to prove Equation \eqref{eq: lrp1} for a matrix $Q$ with a spectral radius strictly less than $1$. For general cases, we consider the loop measure $\mu_{\epsilon}$ associated with $L-\epsilon\cdot Id$, where $Id$ is the identity matrix. Then, Equation \eqref{eq: lrp1} holds for $\mu_{\epsilon}$ and $L-\epsilon\cdot Id$. By taking $\epsilon\downarrow 0$, we get Equation \eqref{eq: lrp1} in the limit. Henceforth, we assume that the eigenvalues $(\lambda_j)_j$ of $Q$ (counted by algebraic multiplicity) are strictly less than $1$. Then, we calculate $\sum\limits_{k\geq 1}\frac{1}{k}\Tr Q^{k}$ by using the eigenvalues:
\begin{equation*}
\sum\limits_{k\geq 1}\frac{1}{k}(\lambda_j)^k=\sum\limits_{j}-\log(1-\lambda_j)=-\log\det(I-Q)=-\log\det(-L)+\sum\limits_{x\in S}\log(-L^x_x).\qedhere
\end{equation*}
\end{proof}

One can deduce the following result from the definition of pointed loop measure. A more general form is hinted in \cite[Exercise 10, Section 2.3]{loop}.

\begin{lem}\label{lem: several closed edges}
 Given a subset $F\subset S\times S$ and a Markovian generator on $S$, we define a modified Markovian generator $\tilde{L}^{F}$ as follows: for two states $x,y\in S$,
 \begin{equation}
 (\tilde{L}^F)^x_y=\left\{
 \begin{array}{ll}
  0 & \text{ if }(x,y)\in F,\\
  L^x_y & \text{ otherwise.}
 \end{array}\right.
 \end{equation}
 Let $\tilde{\mu}$ be the non-trivial pointed loop measure associated with $\tilde{L}$ by Equation \eqref{eq: defn gplm}. Then,
 \begin{equation}
  \mu(\ell\cap F=\emptyset,\mathrm{d}\ell)=\tilde{\mu}(\mathrm{d}\ell).
 \end{equation}
\end{lem}

As we have seen in Lemma \ref{lem: restriction property} and Lemma \ref{lem: several closed edges}, several interesting quantities are related to the determinants of some matrices. For that reason, we state a classical result on the determinants. Please refer to Proposition 2.2 and Example 2.8 in \cite{BottcherMR2179973}.
\begin{lem}[\cite{BottcherMR2179973}] \label{lem: Toeplitz}
Let $T_{3,n}$ be the $n\times n$ tri-diagonal Toeplitz matrix and $S_{n}$ the circulant $n\times n$ matrix such that
$$T_{3,n}=\begin{bmatrix}
a & b & 0 & \cdots & 0\\
c & a & b & \ddots& \vdots\\
0  & \ddots & \ddots & \ddots& 0 \\
 \vdots  & \ddots & c& a& b\\
 0  &\cdots  &  0  & c & a\\
\end{bmatrix}_{n\times n}\text{ and } S_{n}=\begin{bmatrix}
a & b & 0 &  & c\\
c & a & b & \ddots& \\
0  & \ddots & \ddots & \ddots& 0 \\
   & \ddots & c& a& b\\
 b  & &  0  & c & a\\
\end{bmatrix}_{n\times n}.$$

Let $x_1,x_2$ be the roots of $x^2-ax+bc=0$. Then,
\begin{itemize}
\item $\displaystyle{\det(T_{3,n})=\frac{x_1^{n+1}-x_2^{n+1}}{x_1-x_2}}$ for $n\geq 1$,
\item $\det(S_n)=x_1^n+x_2^n+(-1)^{n+1}(b^n+c^n)$ for $n\geq 3$.
\end{itemize}
\end{lem}

Next, we state another useful property of the loop measure: it is ``invariant'' under Doob's harmonic transform. Lemaire and Le Jan have already observed and stated this in the first half part of \cite[Remark 1.1]{Markovian-loop-clusters-on-graphs}. We state it here without the assumption of reversibility of the Markovian generator for the convenience of readers. The proof is immediate from the definition of the loop measure.
\begin{lem}\cite[Remark 1.1]{Markovian-loop-clusters-on-graphs}\label{lem: Doob h transform}
 Suppose that $h:S\rightarrow ]0,\infty[$ is a function on a finite state space $S$ such that $-Lh\geq 0$. Then, $L^{h}$, the Doob's harmonic transformation of $L$, induces the same loop measure, where \begin{equation}
  (L^{h})^x_y\overset{\mathrm{def}}{=}\frac{L^x_yh(y)}{h(x)}\text{ for }x,y\in S.
 \end{equation}
\end{lem}

Finally, we would like to mention that the marginal distributions of the clusters can be expressed by corresponding quantities of the weighted random walk on the graph. As we shall not use the formula, we do not provide the full statement. Please refer to \cite[Lemma 2.7, Lemma 2.8]{Markovian-loop-clusters-on-graphs} for more details.

\section{Loop clusters when no loop passes through \texorpdfstring{$1$}{1}}
We study the discrete loop model on $G^{(n)}$ conditioned on the absence of loops through $1$, and prove Proposition \ref{conditioned renewal process} in Subsection \ref{subsect: Proof of Proposition conditioned renewal process} by identify our model with Model$(\mathbb{Z},\alpha,\kappa^{(n)})$ conditioned on the closedness of $\{1,2\}$ and $\{n,n+1\}$
\subsection{Proof of Proposition \ref{conditioned renewal process}}\label{subsect: Proof of Proposition conditioned renewal process}

\begin{defn}
 For $N=1,2,\ldots,+\infty$ and $\alpha,p,c\geq 0$, denote by Model$([2,N],\alpha,p,c)$ the loop model defined by the loop-soup with intensity measure $\alpha\mu$, where $\mu$ is the loop measure associated with the following Markov generator $L$ by Equation \eqref{eq: defn gplm}:
 $$L^{m}_{m}=-(1+c_n),L^{m}_{m+1}=p_n,L^{m}_{m-1}=1-p_n\text{ for all }m=2,3,\ldots,N,$$
 and $L$ is null elsewhere.
\end{defn}
\begin{rem}
Note that we can identify Model$([2,N],\alpha,p,c)$ and Model$([1,N-1],\alpha,p,c)$ in a natural way by translation invariance.
\end{rem}

By applying Lemma \ref{lem: Doob h transform} with the function $h$ defined by
$$h(m)=\left(\frac{1-p_n}{p_n}\right)^{\frac{m}{2}}\text{ for }m=2,3,\ldots,$$
we see that
\begin{equation}\label{eq: change of parameter by Doob h transform}
 \text{Model}([1,N],\alpha,p,c)=\text{Model}([1,N],\alpha,\kappa)\text{ for }\kappa=\frac{1+c-2\sqrt{p(1-p)}}{\sqrt{p(1-p)}}.
\end{equation}

Among the loop-soup on the discrete circle $G^{(n)}$, the ensemble of loops $\mathcal{DL}_{\alpha,1}^{(n)}$ through the vertex $1$ is independent of its complement $\mathcal{DL}_{\alpha}^{(n)}\setminus\mathcal{DL}_{\alpha,1}^{(n)}$ in the loop-soup. Therefore, $\mathbb{P}[\mathcal{DL}_{\alpha,1}^{(n)}\in\cdot|\mathcal{DL}_{\alpha}^{(n)}\setminus\mathcal{DL}_{\alpha,1}^{(n)}=\emptyset]\overset{\text{law}}{=}\mathcal{DL}_{\alpha,1}^{(n)}$, which also equals in law to the loop-soup in Model$([2,n],\alpha,p^{(n)},c^{(n)})$. By Equation \eqref{eq: change of parameter by Doob h transform}, its law is equal to that of the loop-soup in Model$([2,n],\alpha,\kappa^{(n)})$ where $\kappa^{(n)}=\frac{1+c_n-2\sqrt{p_n(1-p_n)}}{\sqrt{p_n(1-p_n)}}$. Again, by the independence between disjoint loop ensembles, the loop clusters on $G^{(n)}$ has the same distribution as the loop clusters inside $[2,n]$ on Model$(\mathbb{Z},\alpha,\kappa^{(n)})$ by conditioning on the closedness of the edges $\{1,2\}$ and $\{n,n+1\}$. Then, the first part of \cite[Proposition 3.1]{Markovian-loop-clusters-on-graphs} implies Part a) of Proposition \ref{conditioned renewal process}. And Part b) of Proposition \ref{conditioned renewal process} is contained in the proof of \cite[Proposition 3.1]{Markovian-loop-clusters-on-graphs}. For the convenience of the readers, we give a sketch: the jump distribution $\nu^{(\kappa^{(n)})}$ of the renewal process is the distribution of the left end point of the left-most closed edge on $\{1,2,3,\ldots,\}$ in Model$(\mathbb{Z},\alpha,\kappa^{(n)})$, conditionally on the closedness of $\{0,1\}$. Therefore,
$$\mathbb{P}[\{n,n+1\}\text{ is closed}|\{0,1\}\text{ is closed}]=\sum\limits_{k=1}^{\infty}\mathbb{P}[W^{(\kappa^{(n)})}_1+\cdots+W^{(\kappa^{(n)})}_k=n],$$
where $(W^{(\kappa^{(n)})}_i)_{i\geq 1}$ is an independent sequence of variables with the common distribution $\nu^{(\kappa^{(n)})}$. From the expression of $\mathbb{P}[\{n,n+1\}\text{ is closed}|\{0,1\}\text{ is closed}]$ in \cite[Proposition 3.1]{Markovian-loop-clusters-on-graphs}, we get the generating function of the jump distribution $\nu^{(\kappa^{(n)})}$.

\subsection{Proof of Proposition \ref{convergence of conditioned renewal processes}}\label{subsect: proof of convergence of conditioned renewal processes}
As in the statement of Proposition \ref{convergence of conditioned renewal processes}, we assume that $\alpha\in]0,1[$.

For a c\`{a}dl\`{a}g process $X$ and a subset $A$ of the state space, denote by $T_{A}$ the entrance time of $A$, i.e. $T_{A}\overset{\mathrm{def}}{=}\inf\{t\geq 0:X_t\in A\}$. We denote by $X_{t-}$ the left hand limit $\lim\limits_{s\uparrow t}X_s$.

Let $(X^{(\kappa)}_t)_{t\geq 0}$ be a subordinator with potential density $U(x,y)=1_{\{y>x\}}\left(\frac{2\sqrt{\kappa}}{1-e^{-2\sqrt{\kappa}(y-x)}}\right)^{\alpha}$. We first define the law of the process $(X^{(\kappa)}_t,t<T_{]1,+\infty[})$ conditionally on the event $\{X^{(\kappa)}_{T_{]1,+\infty[-}}=1\}$ in the following lemma and postpone its proof in the appendix.

\begin{lem}\ \label{lem: subordinator bridge}
\begin{enumerate}
\item For all positive functions $f$, we have
\begin{multline*}
\mathbb{E}^{0}[f(X^{(\kappa)}_s,s\in[0,t])1_{\{t<T_{]1,+\infty[}\}},X^{(\kappa)}_{T_{]1,+\infty[-}}\in \mathrm{d}b]\\
=\mathbb{E}^{0}[X^{(\kappa)}_{T_{]1,+\infty[-}}\in \mathrm{d}b]\mathbb{E}^{0}\left[f(X^{(\kappa)}_s,s\in[0,t])1_{\{t<T_{]1,+\infty[}\}}\frac{u(b-X^{(\kappa)}_t)}{u(b)}\right].
\end{multline*}
\item The conditioned process\footnote{More precisely, the process defined by the probability $\mathbb{E}^{0}\left[f(X^{(\kappa)}_s,s\in[0,t])1_{\{t<T_{]1,+\infty[}\}}\frac{u(1-X^{(\kappa)}_t)}{u(1)}\right]$.} $Y^{(\kappa)}$ is a $h$-transform of the original subordinator with respect to the excessive function $x\rightarrow u(1-x)$. To be more precise, for $y\in[x,1[$, its semi-group is given by
$$Q^{(\kappa)}_t(x,\mathrm{d}y)=\frac{u(1-y)}{u(1-x)}P^{(\kappa)}_t(x,\mathrm{d}y)$$
where $P^{(\kappa)}_t(x,\mathrm{d}y)$ is the semi-group of the subordinator $X^{(\kappa)}$. Denote by $\mathbb{Q}^{x}$ the law of the Markov process with semi-group $Q^{(\kappa)}_t(x,\mathrm{d}y)=\frac{u(1-y)}{u(1-x)}P^{(\kappa)}_t(x,\mathrm{d}y)$ and initial state $x$. (We choose the c\`{a}dl\`{a}g version of $Y^{(\kappa)}$.)
\item Denote by $\zeta$ the lifetime of the conditioned process $Y^{(\kappa)}$. Then, $Y^{(\kappa)}_{\zeta-}=1$.
\item The semi-group $(Q^{(\kappa)}_t)_{t\geq 0}$ is a Feller semi-group.
\item The time reversal from the lifetime of the process $Y^{(\kappa)}$ is the left-continuous modification of $1-Y^{(\kappa)}$ under $\mathbb{Q}^{0}$.
\item For a fixed $x\in ]0,1[$, with probability $1$ under $\mathbb{Q}^{0}$, it is outside the closure of the range $\bar{\mathcal{R}}(Y^{(\kappa)})$ of $Y^{(\kappa)}$.
\end{enumerate}
\end{lem}

We also need a convergence result of $(\frac{1}{n} S^{(\kappa^{(n)})}_{\lfloor n^{\alpha-1}t\rfloor},t\geq 0)$ as $n\rightarrow\infty$, in the sense of Skorokhod convergence, for a reason that will become clear later. If we put $\epsilon=\frac{1}{n}$ in \cite[Proposition 3.1]{Markovian-loop-clusters-on-graphs}, then their result affirms the convergence of $(\frac{1}{n} S^{(\kappa^{(n)})}_{\lfloor n^{\alpha-1}t\rfloor},t\geq 0)$ towards the subordinator $X^{(\kappa)}$ in the sense of finite marginals convergence, where $\kappa^{(n)}=\kappa/n^2$. Same results hold under the assumption that $\lim\limits_{n\rightarrow\infty}n^2\kappa^{(n)}=\kappa$, and the proof is the same. Indeed, the proof of \cite[Proposition 3.1]{Markovian-loop-clusters-on-graphs} is based on asymptotic behaviors of the Laplace transforms of the jump distributions of the renewal processes. To get this, it suffices to assume that $\lim\limits_{n\rightarrow\infty}n^2\kappa^{(n)}=\kappa$. Also, by a coupling argument, this has been pointed out in Subsection \ref{subsect: connection}:
``the convergence result remains the same if we perturb the killing parameter $\kappa^{(n)}$ (by changing $p_n,c_n$) up to order $o(n^{-2})$''. To strengthen the convergence to a Skorokhod convergence, we show the tightness in the following lemma. The argument is standard and we postpone it in the appendix.
\begin{lem}\label{lem: tightness and skorokhod convergence towards a subordinator}
The distribution $(\frac{1}{n} S^{(\kappa^{(n)})}_{\lfloor n^{\alpha-1}t\rfloor},t\geq 0)$ is tight in the Skorokhod space. Therefore, as $n\rightarrow\infty$, the renewal process $(\frac{1}{n} S^{(\kappa^{(n)})}_{\lfloor n^{\alpha-1}t\rfloor},t\geq 0)$ converges to the subordinator $(X_t^{(\kappa)},t\geq 0)$ in Skorokhod space.
\end{lem}

As a consequence of Lemma \ref{lem: tightness and skorokhod convergence towards a subordinator}, by the coupling theorem of Skorokhod and Dudley, we shall assume that \emph{$(\frac{1}{n}S^{(\kappa^{(n)})}_{\lfloor n^{1-\alpha}t\rfloor},t\geq 0)_n$ converges to $(X^{(\kappa)}_t,t\geq 0)$ almost surely, as $n\rightarrow\infty$}.

\medskip

We are ready for the proof of Proposition \ref{convergence of conditioned renewal processes}. Note that Lemma \ref{lem: subordinator bridge} gives the Radon-Nikodym derivative between the subordinator $X^{(\kappa)}$ and the conditioned process $Y^{(\kappa)}$ on a sub-$\sigma$-field. The main idea of the proof of Proposition \ref{convergence of conditioned renewal processes} is to  show the convergence of the Radon-Nikodym derivatives from the discrete cases to the continuous case.

\bigskip

Firstly, we compute the Radon-Nikodym derivatives between the renewal processes and the conditioned renewal processes. For $m\geq 1$, let $C^{(\kappa^{(n)})}(m)=\mathbb{P}[\exists i\geq 1:S^{(\kappa^{(n)})}_i=m]$. 
For $m\geq 1$, define $T_m=\inf\{i\geq 0: S^{(\kappa^{(n)})}_i\geq m\}$. For all $m\geq 0$ and all positive measurable functions $F:\mathbb{R}^{m+1}\rightarrow\mathbb{R}_{+}$, we have that
\begin{multline}\label{eq: cocrp1}
 \mathbb{E}\left[F(S^{(\kappa^{(n)})}_0,\ldots,S^{(\kappa^{(n)})}_m)1_{\{T_n>m\}}\left|\{\exists i\geq 1:S^{(\kappa^{(n)})}_i=n\}\right.\right]\\
 =\mathbb{E}\left[F(S^{(\kappa^{(n)})}_0,\ldots,S^{(\kappa^{(n)})}_m)1_{\{T_n>m\}}\frac{C^{(\kappa^{(n)})}(n-S^{(\kappa^{(n)})}_m)}{C^{(\kappa^{(n)})}(n)}\right].
\end{multline}
We denote by $({\tilde{\mathcal{G}}_{m}^{(n)})}_m$ the filtration generated by the renewal process $(S^{(\kappa^{(n)})}_i)_i$ and by $(\mathcal{G}^{(n)}_{t})_{t\geq 0}$ the filtration $(\tilde{\mathcal{G}}_{\lfloor n^{1-\alpha}t\rfloor}^{(n)})_{t\geq 0}$. Then, for a stopping time $\tau$ and an event $A\in\mathcal{G}_{\tau}^{(n)}$, from Equation \eqref{eq: cocrp1}, we deduce that
\begin{multline}\label{eq: cocrp2}
 \mathbb{P}\left[A\cap\{S^{(\kappa^{(n)})}_{\lfloor n^{1-\alpha}\tau\rfloor}/n<1\}\left|\{\exists i\geq 1:S^{(\kappa^{(n)})}_i=n\}\right.\right]\\
 =\mathbb{E}\left[1_A1_{\left\{S^{(\kappa^{(n)})}_{\lfloor n^{1-\alpha}\tau\rfloor}/n<1\right\}}\frac{C^{(\kappa^{(n)})}\left(n-S^{(\kappa^{(n)})}_{\lfloor n^{1-\alpha}\tau\rfloor}\right)}{C^{(\kappa^{(n)})}(n)}\right].
\end{multline}

\medskip

Secondly, we will show that for fixed time $t$,
\begin{equation}\label{eq: cocrp3}
 \lim\limits_{n\rightarrow\infty}1_{\left\{S^{(\kappa^{(n)})}_{\lfloor n^{1-\alpha}t\rfloor}/n<1\right\}}\frac{C^{(\kappa^{(n)})}\left(n-S^{(\kappa^{(n)})}_{\lfloor n^{1-\alpha}t\rfloor}\right)}{C^{(\kappa^{(n)})}(n)}=1_{\{X_t<1\}}\frac{u(1-X^{(\kappa)}_t)}{u(1)}.
\end{equation}
By Theorem \ref{thm: lemaire lejan} \cite[Proposition 3.1]{Markovian-loop-clusters-on-graphs}, we have that
\begin{equation*}
 C^{(\kappa^{(n)})}(m)=\left(\frac{1-e^{-2r^{(n)}}}{1-e^{-2(m+1)r^{(n)}}}\right)^{\alpha}.
\end{equation*}
Note that as $n$ tends to $\infty$,
\begin{equation}\label{eq: cocrp4}
 C^{(\kappa^{(n)})}(\lfloor bn\rfloor)\sim \left\{
\begin{array}{ll}
\left(\frac{2\sqrt{\kappa}}{1-e^{-2b\sqrt{\kappa}}}\right)^{\alpha}n^{-\alpha} & \kappa>0,\\
(bn)^{-\alpha} & \kappa=0.
\end{array}
\right.
\end{equation}
Moreover, $(C^{(\kappa^{(n)})}(\lfloor bn\rfloor)n^{\alpha},b\in K)$ converges uniformly on any compact subset $K\subset]0,\infty[$.
As a Feller process, $X^{(\kappa)}$ is continuous at $t$ with probability $1$. Thus, for fixed time $t$, we get that $\lim\limits_{n\rightarrow\infty}\frac{1}{n}S^{(\kappa^{(n)})}_{\lfloor n^{1-\alpha}t\rfloor}=X^{(\kappa)}_t$, which implies Equation \ref{eq: cocrp3}.

\medskip

Finally, we deduce the finite marginals convergence from \eqref{eq: cocrp3}.

For $\delta>0$, $m\geq 1$ and a bounded continuous function $f:\mathbb{R}^m:\rightarrow\mathbb{R}$, the following quantity is uniformly bounded by some finite constant $Cst(\delta,||f||_{\infty})$ for all $n$:
$$\left|f\left(\frac{1}{n}S^{(\kappa^{(n)})}_{\lfloor n^{1-\alpha}s\rfloor},s\in[0, t]\right)1_{\left\{S^{(\kappa^{(n)})}_{\lfloor n^{1-\alpha}t\rfloor}/n<1-\delta\right\}}\frac{C^{(\kappa^{(n)})}\left(n-S^{(\kappa^{(n)})}_{\lfloor n^{1-\alpha}t\rfloor}\right)}{C^{(\kappa^{(n)})}(n)}\right|<Cst(\delta,||f||_{\infty})<\infty.$$
For $0\leq t_1\leq \cdots\leq t_m$, the subordinator $X^{(\kappa)}$, as a Feller process, is almost surely continuous at time $t_1,\ldots,t_m$. Thus, almost surely,
$$\forall i=1,\ldots,m,\quad \lim\limits_{n\rightarrow\infty}\frac{1}{n}S^{(\kappa^{(n)})}_{\lfloor n^{1-\alpha}t_i\rfloor}=X^{(\kappa)}_{t_i}.$$
Thus, by Equation \eqref{eq: cocrp3} and dominated convergence, we get that
\begin{multline*}
\lim\limits_{n\rightarrow\infty}\mathbb{P}\left[f\left(\frac{1}{n}S^{(\kappa^{(n)})}_{\lfloor n^{1-\alpha}t_1\rfloor},\ldots,\frac{1}{n}S^{(\kappa^{(n)})}_{\lfloor n^{1-\alpha}t_m\rfloor}\right)1_{\left\{S^{(\kappa^{(n)})}_{\lfloor n^{1-\alpha}t_m\rfloor}/n<1-\delta\right\}}\frac{C^{(\kappa^{(n)})}\left(n-S^{(\kappa^{(n)})}_{\lfloor n^{1-\alpha}{t_m}\rfloor}\right)}{C^{(\kappa^{(n)})}(n)}\right]\\
=\mathbb{P}\left[f(X^{(\kappa)}_{t_1},\ldots,X^{(\kappa)}_{t_m})\frac{u(1-X^{(\kappa)}_{t_m})}{u(1)},X^{(\kappa)}_{t_m}<1-\delta\right].
\end{multline*}
Equivalently, by Equation \eqref{eq: cocrp1} and Lemma \ref{lem: subordinator bridge}, $\forall \delta>0,m\geq 1,0\leq t_1\leq\cdots\leq t_m$ and bounded continuous $f:\mathbb{R}^m\rightarrow\mathbb{R}$, we have that
\begin{multline*}
\lim\limits_{n\rightarrow\infty}\mathbb{P}\left[f\left(\frac{1}{n}S^{(\kappa^{(n)})}_{\lfloor n^{1-\alpha}t_1\rfloor},\ldots,\frac{1}{n}S^{(\kappa^{(n)})}_{\lfloor n^{1-\alpha}t_m\rfloor}\right)1_{\left\{S^{(\kappa^{(n)})}_{\lfloor n^{1-\alpha}t_m\rfloor}/n<1-\delta\right\}}\left|\{\exists i\geq 1:S^{(\kappa^{(n)})}_i=n\}\right.\right]\\
=\mathbb{P}\left[f(Y^{(\kappa)}_{t_1},\ldots,Y^{(\kappa)}_{t_m}),Y^{(\kappa)}_{t_m}<1-\delta\right].
\end{multline*}
Therefore, we have the uniqueness of all possible sub-sequential limits of the distributions of the finite marginals of the conditioned renewal processes. The scaled conditioned renewal processes are uniformly bounded by $1$, which implies the tightness of all finite marginal distributions. Consequently, we see that Proposition \ref{convergence of conditioned renewal processes} holds in the sense of finite marginals convergence. To get a Skorokhod convergence, we need the tightness of the family of conditioned renewal processes. Note that for $n\geq 1$, a renewal process conditioned to hit $n$ is exchangeable. By the exchangeability, as an application of Aldous' criteria of tightness \cite[Theorem 16.11]{KallenbergMR1876169}, the finite marginals convergence implies the tightness, see the proof of \cite[Theorem 16.23]{KallenbergMR1876169}.

\section{Loop clusters when all loops pass through \texorpdfstring{$1$}{1}}\label{sec: loops through 1}
In this section, we will prove Proposition \ref{loops through 1}, a description of the loop clusters conditionally on the absence of the loops $\mathcal{DL}_{\alpha,1}^{(n)}$ avoiding the vertex $1$. As we have mentioned in the introduction, we will divide the loop-soup $\mathcal{DL}_{\alpha}^{(n)}\setminus\mathcal{DL}_{\alpha,1}^{(n)}$ passing through $1$ into three disjoint loop-soups: $\mathcal{DL}_{\alpha,2}^{(n)}$, $\mathcal{DL}_{\alpha,3}^{(n)}$ and $\mathcal{DL}_{\alpha,4}^{(n)}$. As a loop-soup is a Poisson point process, $\mathcal{DL}_{\alpha,2}^{(n)}$, $\mathcal{DL}_{\alpha,3}^{(n)}$ and $\mathcal{DL}_{\alpha,4}^{(n)}$ are independent Poisson point process. We will study them separately and then put the results together to prove Proposition \ref{loops through 1}. To precise the definition of $\mathcal{DL}_{\alpha,2}^{(n)}$, $\mathcal{DL}_{\alpha,3}^{(n)}$ and $\mathcal{DL}_{\alpha,4}^{(n)}$, we need to introduce several notation.

We know that $\mathbb{Z}$ is a covering space of $G^{(n)}$ under the following mapping $\pi^{(n)}$:
$$\pi^{(n)}(i+kn)=i+1\text{ for }k\in\mathbb{Z}\text{ and }i=0,\ldots,n-1.$$
\begin{defn}\label{defn: rotation number and lifted pointed loop}
When $\gamma=(\gamma(1),\ldots,\gamma(m))$ is a path in $G^{(n)}$ and $z$ is a point ``lying over'' $\gamma(1)$ (i.e. $\pi^{(n)}(z)=\gamma(1)$), then there exists a unique path $\Gamma$ in $\mathbb{Z}$ lying over $\gamma$ (i.e. $\pi^{(n)}\circ\Gamma=\gamma$) such that $\Gamma(1)=z$. The path $\Gamma$ is called the lift of $\gamma$ at $z$.

By definition, a pointed loop $\dot{\ell}=(x_1,x_2,\ldots,x_m)$ on $G^{(n)}$ is a path $\gamma=(x_1,x_2\ldots,x_m,x_1)$ on $G^{(n)}$. For $z\in\mathbb{Z}$ such that $\pi^{(n)}(z)=x_1$, let $\Gamma=(\Gamma(1),\ldots,\Gamma(m+1))$ be the lift of $\gamma$ at $z$. Then, $(\Gamma(m+1)-\Gamma(1))/n$ is an integer independent of the choice of $z$, which is defined to be the rotation number $\Rot(\dot{\ell})$ of the pointed loop $\dot{\ell}$. If we choose $z$ within $\{0,1,\ldots,n-1\}$, then $\Gamma$ is uniquely determined. When $\Rot(\dot{\ell})=0$, we have that $\Gamma(m+1)=\Gamma(1)$ and $\Gamma$ is a bridge. Then, we denote by $\Lift(\dot{\ell})$ the unique pointed loop $(\Gamma(1),\ldots,\Gamma(m))$, by choosing $\Gamma(1)\in\{0,1,\ldots,n-1\}$. Since two equivalent pointed loops have the same rotation number, the rotation number $\Rot(\ell)$ of a loop $\ell$ is well-defined.
\end{defn}

\begin{defn}\label{defn:lift loop}
 Define a $0$-$1$ valued function $\Psi^{(n)}$ on non-trivial loops on $G^{(n)}$: $\Psi^{(n)}(\ell)=1$ iff the following three conditions are all fulfilled.
 \begin{itemize}
  \item[a)] $\Rot(\ell)=0$,
  \item[b)] $\ell$ passes through the vertex $1$ in the discrete circle $G^{(n)}$,
  \item[c)] suppose that $\dot{\ell}$ and $\dot{\ell}'$ are both in the equivalence class $\ell$ and start from the vertex $1$ in $G^{(n)}$. According to Definition \ref{defn: rotation number and lifted pointed loop}, we have a unique pointed loop $\Lift(\dot{\ell})$ on $\mathbb{Z}$ starting from $0$ as the lift of $\dot{\ell}$. Similarly, we get $\Lift(\dot{\ell}')$. Then, our condition c) requires that $\Lift(\dot{\ell})$ and $\Lift(\dot{\ell}')$ are equivalent pointed loops.
 \end{itemize}
For a loop $\ell$ such that $\Psi^{(n)}(\ell)=1$, we choose some representative pointed loop $\dot{\ell}$ in the equivalence class $\ell$. Denote by $\Lift(\dot{\ell})$ the unique pointed loop on $\mathbb{Z}$ starting from $0$ that lies over $\dot{\ell}$. Then, we define the lift of the loop $\ell$ to be the loop $\Lift(\ell)$ which is the equivalence class of $\Lift(\dot{\ell})$. 
\end{defn}
\begin{rem}
 Condition $c)$ is equivalent to the following statement: let $\dot{\ell}=(x_1,\ldots,x_m)$ be a pointed loop in the class $\ell$, starting from the vertex $x_1=1$, with zero rotation number. Then, there exists no consecutive subsequence $1,2,\ldots,n,1$ or $1,n,n-1,\ldots,1$ inside $x_1,\ldots,x_n$. 
\end{rem}

We introduce Definition \ref{defn:lift loop} for the following purpose:
\begin{lem}\label{lem:lift loop measure}
 Let $\mu_{n,\mathbb{Z}}$ be the non-trivial loop measure on $\mathbb{Z}$ associated with the following Markov generator $L$:
$$L^{i}_{j}=\left\{
\begin{array}{ll}
\frac{1}{2} & \text{ for }|i-j|=1,\\
-1-\kappa^{(n)}/2 & \text{ for }i=j,\\
0 & \text{otherwise,}
\end{array}
\right.$$
where $\kappa^{(n)}=\frac{1+c_n-2\sqrt{p_n(1-p_n)}}{\sqrt{p_n(1-p_n)}}$. Then, the push-forward $\Lift\circ\mu_n(\Psi^{(n)}(\ell)=1,\mathrm{d}\ell)$ of the measure $\mu_n(\Psi^{(n)}(\ell)=1,\mathrm{d}\ell)$ equals
$$\mu_{n,\mathbb{Z}}(0\in\ell\text{ and }\ell\subset[1-n,n-1],\mathrm{d}\ell).$$
\end{lem}
\begin{proof}
 Lemma \ref{lem:lift loop measure} can be proven by comparing the weights of each particular loop under these two measures.
\end{proof}

Next, we define a partition $(\mathcal{O}^{(n)}_i,i=1,2,3,4)$ of possible non-trivial loops on $n$-th discrete circle $G^{(n)}$.

\begin{defn}\label{defn: four type of loops}
Let $\mathcal{O}^{(n)}_1$ be the ensemble of non-trivial loops avoiding $1$, $\mathcal{O}^{(n)}_2$ the ensemble of non-trivial loops passing through $1$ with non-zero rotation numbers, $\mathcal{O}^{(n)}_3$ the ensemble $\{\ell\text{ is non-trivial}:\Phi^{(n)}(\ell)=1\}$, and $\mathcal{O}^{(n)}_4$ the remainder. Let $\mathcal{O}^{(n)}_{\text{cov}}$ be the ensemble of non-trivial loops which cover all the vertices in the discrete circle $G^{(n)}$. Define $\mathcal{DL}^{(n)}_{\alpha,i}=\mathcal{DL}_{\alpha}^{(n)}\cap\mathcal{O}^{(n)}_i$ for $i=1,2,3,4$, where $\mathcal{DL}_{\alpha}^{(n)}$ is the Poissonian loop ensemble of intensity $\alpha\mu_n$.
\end{defn}
By the definition of the Poisson random measure, $(\mathcal{DL}_{\alpha,i}^{(n)},i=1,2,3,4)$ are independent. Also, note that $\mathcal{DL}^{(n)}_{\alpha,1}\cap\mathcal{O}^{(n)}_{\text{cov}}=\emptyset$ and that $\mathcal{DL}^{(n)}_{\alpha,2}\cup\mathcal{DL}^{(n)}_{\alpha,4}\subset\mathcal{O}^{(n)}_{\text{cov}}\subset \mathcal{DL}^{(n)}_{\alpha,2}\cup\mathcal{DL}^{(n)}_{\alpha,3}\cup\mathcal{DL}^{(n)}_{\alpha,4}$. From the definition of the non-trivial loop measure $\mu_n$, the law of $\mathcal{DL}^{(n)}_{\alpha,1}$, $\mathcal{DL}^{(n)}_{\alpha,3}$ and $\mathcal{DL}^{(n)}_{\alpha,4}$ do not change if we replace $(p_n,1-p_n,1+c_n)$ by $(\sqrt{p_n(1-p_n)},\sqrt{p_n(1-p_n)},1+c_n)$ (or $(\frac{1}{2},\frac{1}{2},1+\frac{\kappa^{(n)}}{2})$ equivalently). The non-symmetry only affects the distribution of $\mathcal{DL}^{(n)}_{\alpha,2}$. (This will become clear in the following subsections.)

\bigskip

We will study $\mathcal{DL}^{(n)}_{\alpha,3}$ and $\mathcal{DL}^{(n)}_{\alpha,2}\cup\mathcal{DL}^{(n)}_{\alpha,4}$ in different subsections. Then, we will prove Proposition \ref{loops through 1} in the last subsection of the present section.

\subsection{Loop-soup \texorpdfstring{$\mathcal{DL}^{(n)}_{\alpha,2}\cup\mathcal{DL}^{(n)}_{\alpha,4}$}{DLα,2 ∪ DLα,4}}
Note that the loops in $\mathcal{DL}^{(n)}_{\alpha,2}\cup\mathcal{DL}^{(n)}_{\alpha,4}$ are loops covering all the vertices in $G^{(n)}$. To study the loop clusters, it suffices to calculate the probability $\mathbb{P}[\mathcal{DL}^{(n)}_{\alpha,2}\cup\mathcal{DL}^{(n)}_{\alpha,4}=\emptyset]$, which is given by the following lemma.
\begin{lem}\label{lem: absence of L2 and L4}
Suppose that $r^{(n)}$ is the same as in Definition \ref{defn: notation}. We have that
$$\mathbb{P}[\mathcal{DL}^{(n)}_{\alpha,2}\cup\mathcal{DL}^{(n)}_{\alpha,4}=\emptyset]=\left(\frac{\cosh(nr^{(n)})}{\cosh(nr^{(n)})-\cosh(n\log(\frac{p_n}{1-p_n})/2)}\right)^{-\alpha}.$$
If $\lim\limits_{n\rightarrow\infty}n^2\kappa^{(n)}=\kappa$ and $\lim\limits_{n\rightarrow\infty}n^2c_n=\epsilon\in[0,\kappa/2]$, then
\begin{equation}\label{eq: laoLaL0}
 \lim\limits_{n\rightarrow\infty}\mathbb{P}[\mathcal{DL}^{(n)}_{\alpha,2}\cup\mathcal{DL}^{(n)}_{\alpha,4}=\emptyset]=\left(\frac{\cosh(\sqrt{\kappa})-\cosh(\sqrt{\kappa-2\epsilon})}{\cosh(\sqrt{\kappa})}\right)^{\alpha}.
\end{equation}
\end{lem}
\begin{proof}
Since $\mathcal{DL}_{\alpha,2}^{(n)}\cup\mathcal{DL}_{\alpha,4}^{(n)}$ is a Poisson point process,
\begin{equation}\label{eq: laoLaL1}
 \mathbb{P}[\mathcal{DL}^{(n)}_{\alpha,2}\cup\mathcal{DL}^{(n)}_{\alpha,4}=\emptyset]=\exp\{-\alpha\mu_n(\mathcal{O}^{(n)}_2\cup\mathcal{O}^{(n)}_4)\},
\end{equation}
where $\mu_n$ is the push-forward measure of $\dot{\mu}_n$ defined in Equation \eqref{eq: 1}, and
\begin{equation}\label{eq: laoLaL2}
 \mu_n(\mathcal{O}^{(n)}_2\cup\mathcal{O}^{(n)}_4)=\mu_n(1\in\ell)-\mu_n(\mathcal{O}^{(n)}_3).
\end{equation}
Let's calculate $\mu_n(1\in\ell)$: By Lemma \ref{lem: restriction property}, we have that
\begin{align*}
 \mu_n(1)=&-\log\det(-L^{(n)})+\sum\limits_{i=1}^{n}\log(-L^{(n)})^i_i\\
 \mu_n(\ell\subset\{2,\ldots,n\})=&-\log\det(-L^{(n)}|_{\{2,\ldots,n\}^2})+\sum\limits_{i=2}^{n}\log(-L^{(n)})^i_i.
\end{align*}
Thus, by taking the difference, we see that
\begin{align*}
\mu_n(1\in\ell)=&\mu_n(1)-\mu_n(\ell\subset\{2,\ldots,n\})\\
=&\log(-L^{(n)})^1_1+\log(\det(-L^{(n)}|_{\{2,\ldots,n\}^2}))-\log(\det(-L^{(n)})).
\end{align*}
By Lemma \ref{lem: Toeplitz} for the determinants, the above quantity equals $$\log(1+c_n)+\log(x_1^n-x_2^n)-\log(x_1-x_2)-\log(x_1^n+x_2^n-p_n^n-(1-p_n)^n),$$
where $x_1=e^{r^{(n)}}\sqrt{p_n(1-p_n)}$ and $x_2=e^{-r^{(n)}}\sqrt{p_n(1-p_n)}$. Or equivalently,
\begin{align}\label{eq: laoLaL3}
\mu_n(1\in\ell)=&\log\left(\frac{1+c_n}{\sqrt{(1+c_n)^2-4p_n(1-p_n)}}\right)+\log(\sinh(nr^{(n)}))\notag\\
&-\log\left(\cosh(nr^{(n)})-\cosh(n\log(p_n/(1-p_n))/2)\right).
\end{align}
Next, we calculate $\mu_n(\mathcal{O}^{(n)}_3)$: By Lemma \ref{lem:lift loop measure},
\begin{align*}
\mu_n(\mathcal{O}^{(n)}_3)=&\mu_{n,\mathbb{Z}}(0\in\ell,\ell\subset[1-n,n-1])\\
=&\mu_{n,\mathbb{Z}}(\ell\subset[1-n,n-1])-\mu_{n,\mathbb{Z}}(\ell\subset[1-n,-1])-\mu_{n,\mathbb{Z}}(\ell\subset[1,n-1]),
\end{align*}
where $\mu_{n,\mathbb{Z}}$ is defined in Lemma \ref{lem:lift loop measure}. By Lemma \ref{lem: restriction property} and Lemma \ref{lem: Toeplitz},
\begin{equation}\label{eq: laoLaL4}
\mu_n(\mathcal{O}^{(n)}_3)=\log\left(\frac{1+c_n}{\sqrt{(1+c_n)^2-4p_n(1-p_n)}}\right)+\log(\tanh(nr^{(n)})).
\end{equation}
By combining Equations \eqref{eq: laoLaL1}, \eqref{eq: laoLaL2}, \eqref{eq: laoLaL3} and \eqref{eq: laoLaL4} together,
\begin{equation*}
 \mathbb{P}[\mathcal{DL}^{(n)}_{\alpha,2}\cup\mathcal{DL}^{(n)}_{\alpha,4}=\emptyset]=\left(\frac{\cosh(nr^{(n)})}{\cosh(nr^{(n)})-\cosh(n\log(\frac{p_n}{1-p_n})/2)}\right)^{-\alpha}.
\end{equation*}
Under the assumptions $\lim\limits_{n\rightarrow\infty}n^2\kappa^{(n)}=\kappa$ and $\lim\limits_{n\rightarrow\infty}n^2c_n=\epsilon\in[0,\kappa/2]$, we have
$$\frac{p_n}{1-p_n}=1-\frac{2\sqrt{\kappa-2\epsilon}}{n}+o(1/n)\text{ and }r^{(n)}=\frac{\sqrt{\kappa}}{n}+o(1/n).$$ 
Consequently,
\begin{equation*}
 \lim\limits_{n\rightarrow\infty}\mathbb{P}[\mathcal{DL}^{(n)}_{\alpha,2}\cup\mathcal{DL}^{(n)}_{\alpha,4}=\emptyset]=\left(\frac{\cosh(\sqrt{\kappa})-\cosh(\sqrt{\kappa-2\epsilon})}{\cosh(\sqrt{\kappa})}\right)^{\alpha}. \qedhere
\end{equation*}
\end{proof}

\subsection{Loop-soup \texorpdfstring{$\mathcal{DL}_{\alpha,3}^{(n)}$}{DLα,3}}\label{subsect: dl3}
For a loop $\ell$ in $\mathcal{DL}^{(n)}_{\alpha,3}$, $\Lift(\ell)$ is a loop on $\mathbb{Z}$ passing through $0$ but never reaching $-n$ nor $n$. By Lemma \ref{lem:lift loop measure}, $\Lift(\mathcal{DL}^{(n)}_{\alpha,3})\overset{\mathrm{def}}{=}\{\Lift(\ell):\ell\in\mathcal{DL}_{\alpha,3}^{(n)}\}$ is the Poisson ensemble of loops on $\mathbb{Z}$ with intensity measure $\alpha\mu_{n,\mathbb{Z}}$. They cover a discrete random sub-interval $[-A_n,B_n]$ of $[-n+1,\ldots,n-1]$ which contains $0$. When $A_n+B_n\geq n-1$, all the vertices belong to the same loop cluster formed by $\mathcal{DL}_{\alpha,3}^{(n)}$; when $A_n+B_n\leq n-2$, there is a correspondence between $[-A_n,B_n]$ and the random discrete arc covered by $\mathcal{DL}^{(n)}_{\alpha,3}$ such that $A_n=J_n$ and $B_n=K_n$. We give the distribution of $[-A_n,B_n]$ in the following lemma.
\begin{lem}\label{lem: limit of the loops passing 1}
For a fixed sub-interval $[-m_n,M_n]$ in $[1-n,n-1]$, we have that
\begin{multline*}
 \mathbb{P}([-A_n,B_n]\subset [-m_n,M_n])\\
 =\left(\frac{2\cosh nr^{(n)}}{\sinh nr^{(n)}}\right)^{\alpha}\left(\frac{\sinh \left((m+1)r^{(n)}\right)\sinh \left((M+1)r^{(n)}\right)}{\sinh (m+M+2)r^{(n)}}\right)^{\alpha},
\end{multline*}
where $r^{(n)}$ is given in Definition \ref{defn: notation}. As $n$ tends to infinity, under the assumption that $\lim\limits_{n\rightarrow\infty}n^2\kappa^{(n)}=\kappa$, the sequence of variables $(\frac{A_n}{n},\frac{B_n}{n})_n$ converges in distribution towards $(A,B)\in[0,1]^2$ where
$$\mathbb{P}[A\leq  a,B\leq b]=\left(\frac{2\cosh(\sqrt{\kappa})}{\sinh(\sqrt{\kappa
})}\right)^{\alpha}\left(\frac{\sinh(\sqrt{\kappa}a)\sinh(\sqrt{\kappa}b)}{\sinh(\sqrt{\kappa}(a+b))}\right)^{\alpha}.$$
\end{lem}
\begin{proof}
We fix a sub-interval $[-m_n,M_n]$ of $[1-n,n-1]$. Then,
\begin{multline}\label{eq: llotlp11}
\mathbb{P}([-A_n,B_n]\subset [-m_n,M_n])\\
=\exp\{-\alpha\mu_{n,\mathbb{Z}}(0\in\ell, \ell\subset [1-n,n-1],\ell\not\subset[-m_n,M_n])\},
\end{multline}
where $\mu_{n,\mathbb{Z}}$ is defined in Lemma \ref{lem:lift loop measure}. By inclusion-exclusion principle, for positive integers $m_n$ and $M_n$, we have that
\begin{multline*}
\mu_{n,\mathbb{Z}}(0\in\ell, \ell\subset[-m_n,M_n])\\
=\mu_{n,\mathbb{Z}}(\ell\subset [-m_n,M_n])-\mu_{n,\mathbb{Z}}(\ell\subset [-m_n,-1])-\mu_{n,\mathbb{Z}}(\ell\subset[1,M_n]).
\end{multline*}
By Lemma \ref{lem: restriction property}, we see that
$$\mu_{n,\mathbb{Z}}(\ell\subset [-m_n,M_n])=-\log\det(-L|_{[-m_n,M_n]})+\sum\limits_{x=-m_n}^{M_n}\log(-L^x_x),$$
where $L$ is the Markov generator defined in Lemma \ref{lem:lift loop measure}. Similar expressions hold for the terms $\mu_{n,\mathbb{Z}}(\ell\subset [-m_n,-1])$ and $\mu_{n,\mathbb{Z}}(\ell\subset [1,M_n])$.
Thus,
\begin{multline*}
\mu_{n,\mathbb{Z}}(0\in\ell, \ell\subset [-m_n,M_n])\\
=\log(-L^{0}_{0})-\log\det(-L|_{[-m_n,M_n]})+\log\det(-L|_{[-m_n,-1]})+\log\det(-L|_{[1,M_n]}).
\end{multline*}
We calculate the above determinants by using Lemma \ref{lem: Toeplitz}:
\begin{align*}
(-L^0_0)=&\frac{x_1^2-x_2^2}{x_1-x_2},\\
\det(-L|_{[-m_n,M_n]})=&\frac{x_1^{M_n+m_n+2}-x_2^{M_n+m_n+2}}{x_1-x_2},\\
\det(-L|_{[-m_n,-1]})=&\frac{x_1^{m_n+1}-x_2^{m_n+1}}{x_1-x_2},\\
\det(-L|_{[1,M_n]})=&\frac{x_1^{M_n+1}-x_2^{M_n+1}}{x_1-x_2},
\end{align*}
where $x_1$ and $x_2$ are the solutions of $x^2-(1+c_n)x+\sqrt{p_n(1-p_n)}=0$:
\begin{equation}\label{eq: llotlp12}
x_1=e^{r^{(n)}}\sqrt{p_n(1-p_n)}\text{ and }x_2=e^{-r^{(n)}}\sqrt{p_n(1-p_n)}.
\end{equation}
Therefore,
\begin{equation*}
\mu_{n,\mathbb{Z}}(0\in\ell, \ell\subset[-m_n,M_n])=\log\left((x_1+x_2)\frac{(x_1^{M_n+1}-x_2^{M_n+1})(x_1^{m_n+1}-x_2^{m_n+1})}{(x_1-x_2)(x_1^{m_n+M_n+2}-x_2^{m_n+M_n+2})}\right).
\end{equation*}
In particular,
\begin{equation*}
\mu_{n,\mathbb{Z}}(0\in\ell, \ell\subset [-n+1,n-1])=\log\left((x_1+x_2)\frac{(x_1^{n}-x_2^{n})(x_1^{n}-x_2^{n})}{(x_1-x_2)(x_1^{2n}-x_2^{2n})}\right).
\end{equation*}
By taking the difference, we see that
\begin{multline}\label{eq: llotlp13}
\mu_{n,\mathbb{Z}}(0\in\ell, \ell\subset[-n+1,n-1], \ell\not\subset[-m_n,M_n])\\
=\log\left(\frac{(x_1^{n}-x_2^{n})(x_1^{n}-x_2^{n})(x_1^{m_n+M_n+2}-x_2^{m_n+M_n+2})}{(x_1^{M_n+1}-x_2^{M_n+1})(x_1^{m_n+1}-x_2^{m_n+1})(x_1^{2n}-x_2^{2n})}\right).
\end{multline}
By combining Equations \eqref{eq: llotlp11}, \eqref{eq: llotlp12} and \eqref{eq: llotlp13}, we get that
\begin{equation*}
\mathbb{P}([-A_n,B_n]\subset [-m_n,M_n])=\left(\frac{2\cosh nr^{(n)}}{\sinh nr^{(n)}}\cdot\frac{\sinh \left((m+1)r^{(n)}\right)\sinh \left((M+1)r^{(n)}\right)}{\sinh (m+M+2)r^{(n)}}\right)^{\alpha}. 
\end{equation*}
Finally, by an explicit calculation, we get the convergence result for $(\frac{A_n}{n},\frac{B_n}{n})$, as $n\rightarrow\infty$, under the assumption that $\lim\limits_{n\rightarrow\infty}n^2\kappa^{(n)}=\kappa$.
\end{proof}

\subsection{Proof of Proposition \ref{loops through 1}}
In this subsection, we combine the previous results and give a proof of Proposition \ref{loops through 1}. Note that for non-negative integers $m$ and $M$ such that $m+M\leq n-2$,
\begin{multline*}
 \{\exists \geq 2\text{ loop clusters},J_n\leq m,K_n\leq M,\mathcal{DL}_{\alpha,1}^{(n)}=\emptyset\}\\
 =\{\mathcal{DL}_{\alpha,1}\cup\mathcal{DL}_{\alpha,2}\cup\mathcal{DL}_{\alpha,4}=\emptyset, A_n\leq m,B_n\leq M\},
\end{multline*}
where $A_n,B_n$ are defined in Subsection \ref{subsect: dl3}. Then, by the independence of $(\mathcal{DL}_{\alpha,i}^{(n)})_{i=1,2,3,4}$, Lemmas \ref{lem: absence of L2 and L4} and \ref{lem: limit of the loops passing 1}, we get that
\begin{align*}
 \mathbb{P}[\exists \geq 2\text{ loop clusters},& J_n\leq m,K_n\leq M|\mathcal{DL}_{\alpha,1}^{(n)}=\emptyset]\\
 =&\mathbb{P}[\mathcal{DL}_{\alpha,2}^{(n)}\cup\mathcal{DL}_{\alpha,4}^{(n)}=\emptyset]\mathbb{P}[A_n\leq m,B_n\leq M]\\
 =&2^{\alpha}\left(\frac{\cosh (nr^{(n)})-\cosh\left(\frac{1}{2}n\log\left(\frac{p_n}{1-p_n}\right)\right)}{\sinh (nr^{(n)})}\right)^{\alpha}\\
  &\times\left(\frac{\sinh ((m+1)r^{(n)})\sinh ((M+1)r^{(n)})}{\sinh ((m+M+2)r^{(n)})}\right)^{\alpha},
\end{align*}
which implies the expression of $\mathbb{P}[\exists \geq 2\text{ loop clusters}|\mathcal{DL}_{\alpha,1}^{(n)}=\emptyset]$. The limit result in Proposition \ref{loops through 1} is a consequence of the limit result in Lemma \ref{lem: limit of the loops passing 1}. Indeed,
as the limit distribution $(A,B)$ has a probability density, we have that
\begin{equation*}
 \lim\limits_{n\rightarrow\infty}\mathbb{P}[\exists \geq 2\text{ loop clusters}|\mathcal{DL}_{\alpha,1}^{(n)}=\emptyset]=\mathbb{P}[A+B\leq 1],
\end{equation*}
and for positive real numbers $a$ and $b$ such that $a+b\leq 1$,
\begin{equation*}
 \lim\limits_{n\rightarrow\infty}\mathbb{P}[\exists \geq 2\text{ loop clusters}, J_n\leq an,K_n\leq bn|\mathcal{DL}_{\alpha,1}^{(n)}=\emptyset]=\mathbb{P}[A\leq a,B\leq b].
\end{equation*}

\section{Proof of Theorem \ref{thm: limit distribution of loop cluster}}
For Theorem \ref{thm: limit distribution of loop cluster}, it suffices to prove the following lemma. We will explain this in details after the statement of the lemma.
\begin{lem}\label{lem: the cluster at 0}
For $0<\alpha<1$,
\begin{equation}\label{eq: ltca01}
\lim\limits_{n\rightarrow\infty}\mathbb{P}[\exists\geq 2\text{ loop clusters}|\mathcal{DL}^{(n)}_{\alpha,2}\cup\mathcal{DL}^{(n)}_{\alpha,4}=\emptyset]=\frac{(2\cosh\sqrt{\kappa})^{\alpha}\sinh(\sqrt{\kappa}(1-\alpha))}{\sinh\sqrt{\kappa}}.
\end{equation}
Conditionally on the existence of closed edges, $(G_n/n,D_n/n)$ converges in distribution towards $G,D$ where the density $q(x,y)$ of $(G,D)$ is given by
\begin{multline*}
 \mathbb{P}[G\in \mathrm{d}x,D\in \mathrm{d}y]/\mathrm{d}x\mathrm{d}y\\
=\frac{\sin(\alpha\pi)}{\pi}\frac{2^{\alpha-2}(1-\alpha)\kappa\sinh\sqrt{\kappa}}{\sinh(\sqrt{\kappa}(1-\alpha))\left[\sinh(\sqrt{\kappa}(1-x-y))\right]^{\alpha}\left[\sinh(\sqrt{\kappa}(x+y))\right]^{2-\alpha}}.
\end{multline*}
\end{lem}
\begin{proof}[Proof of Theorem \ref{thm: limit distribution of loop cluster} by using Lemma \ref{lem: the cluster at 0}]\ 

 By independence of $(\mathcal{DL}_{\alpha,i}^{(n)})_{i=1,2,3,4}$, Equations \eqref{eq: laoLaL0} and \eqref{eq: ltca01} imply Part a) of Theorem \ref{thm: limit distribution of loop cluster} for $\alpha\in]0,1[$. For $\alpha\geq 1$, since $\mathbb{P}[\#\mathcal{C}_{\alpha}^{(n)}=1]$ increases as $\alpha$ increases, the result is obtained by taking $\alpha\uparrow 1$. Since $\mathbb{P}[(G_n/n,D_n/n)\in \cdot|\exists \text{ closed edges}]$ converges towards $(G,D)$ and the distribution of $G+D$ has no atom, we must have
$$\lim\limits_{n\rightarrow\infty}\mathbb{P}[G_n+D_n=n-1|\exists \text{ closed edges}]=0.$$
As a result,
\begin{equation*}
 \mathbb{P}[\exists \text{ a unique closed edge in }G^{(n)}]\leq\mathbb{P}[G_n+D_n=n-1|\exists \text{ closed edges}]\overset{n\rightarrow\infty}{\rightarrow}0.
\end{equation*}
Therefore, by Lemma \ref{lem: the cluster at 0}, we also have the convergence of $(G_n/n,D_n/n)$ by conditioning on the event $\{\#\mathcal{C}_{\alpha}^{(n)}\geq 2\}$ :
$$\lim\limits_{n\rightarrow\infty}\mathbb{P}[(G_n/n,D_n/n)\in\cdot|\#\mathcal{C}_{\alpha}^{(n)}\geq 2]=\mathbb{P}[(G,D)\in\cdot].$$
By the coupling theorem of Skorokhod and Dudley, we shall assume that \emph{$(G_n,D_n)_n$ converges almost surely to $(G,D)$ as $n\rightarrow\infty$}. As in the statement of Theorem \ref{thm: limit distribution of loop cluster}, the processes $(S_i^{(n)}-S_0^{(n)})_{i=0,\ldots,k(n)}$ are the left end points of closed edges, shifted by $S_0^{(n)}=D_n+1$. Conditionally on $(G_n,D_n)$, the sequence $(S_i^{(n)}-S_0^{(n)})_{i=0,\ldots,k(n)}$ has the same law as the conditioned renewal processes formed by the left end points of the closed edges in Model$([1,\ldots,n-1-G_n-D_n],\alpha,\kappa^{(n)})$. Note that
$$\lim\limits_{n\rightarrow\infty}(n-1-G_n-D_n)^2\kappa^{(n)}=(1-G-D)^2\kappa.$$ Therefore, by Proposition \ref{convergence of conditioned renewal processes},
\begin{equation*}
 \lim\limits_{n\rightarrow\infty}\mathbb{P}[\tilde{S}^{(n)}\in\cdot|G_n,D_n]\overset{\text{Skorokhod}}{=}\mathbb{P}[Y^{(\kappa(1-G-D)^{2})}\in\cdot |G,D],
\end{equation*}
where the scaled process $\tilde{S}^{(n)}_t=\frac{1}{n-1-G_n-D_n}(S^{(n)}_{\lfloor (n-1-G_n-D_n)^{1-\alpha}t\rfloor}-S^{(n)}_0)$.
\end{proof}

To prove Lemma \ref{lem: the cluster at 0}, we need the following lemmas.
\begin{lem}\label{lem: levy measure}
The L\'{e}vy measure $\Pi$ of the subordinator of the renewal density $u(x)=(\frac{2\sqrt{\kappa}}{1-e^{-2\sqrt{\kappa}x}})^{\alpha}$ is given by the following expression:
$$\Pi(\mathrm{d}t)=\,\mathrm{d}t\cdot \frac{1}{\pi}(1-\alpha)\sin(\alpha\pi)e^{2\sqrt{\kappa}(\alpha-1)t}\left(\frac{2\sqrt{\kappa}}{1-e^{-2\sqrt{\kappa}t}}\right)^{2-\alpha}.$$
\end{lem}

\begin{lem}\label{lem:distribution of XTa}
Consider the subordinator $X^{(\kappa)}$ of the potential density $u(x)=\left(\frac{2\sqrt{\kappa}}{1-e^{-2\sqrt{\kappa}x}}\right)^{\alpha}$. For $a>0$, we have that $\mathbb{P}^{0}[X^{(\kappa)}_{T_{]a,\infty[}}=a]=0$ and that
\begin{align*}
\mathbb{P}^{0}[X^{(\kappa)}_{T_{]a,\infty[}}\in \mathrm{d}x]/\mathrm{d}x=&\int\limits_{z\in ]x-a,x[}u(x-z)\Pi(\mathrm{d}z)\\
=&\frac{\sqrt{\kappa}}{\pi}\sin(\alpha\pi)\frac{e^{\alpha\sqrt{\kappa}x}(\sinh(\sqrt{\kappa}a))^{1-\alpha}}{\sinh(\sqrt{\kappa}x)(\sinh(\sqrt{\kappa}(x-a)))^{1-\alpha}}.
\end{align*}
\end{lem}
\begin{proof}
The subordinator $X^{(\kappa)}$ has zero drift as $\lim\limits_{x\downarrow 0}U(0,x)=\infty$ by \cite[Theorem 5 in Chapter 3]{BertoinMR1406564}. Consequently, for any fixed $a>0$, by a result of H. Kesten \cite{KestenMR0272059} (see \cite[Proposition 1.9 (i)]{BertoinMR1746300}), for $a>0$,
$$\mathbb{P}^{0}[a\text{ belongs to the closure of the range of }X^{(\kappa)}]=0.$$
Hence, $\mathbb{P}^{0}[X^{(\kappa)}_{T_{]a,\infty[}}=a]=0$ for $a>0$. According to Lemma 1.10 in \cite{BertoinMR1746300}, for $x>a$,
$$\mathbb{P}^{0}[X^{(\kappa)}_{T_{]a,\infty[}}\in \mathrm{d}x]/\mathrm{d}x=\int\limits_{z\in ]x-a,x[}u(x-z)\Pi(\mathrm{d}z).$$
By Lemma \ref{lem: levy measure},
$$\Pi(\mathrm{d}z)=\,\mathrm{d}z\cdot\frac{1}{\pi}(1-\alpha)\sin(\alpha\pi)e^{2\sqrt{\kappa}(\alpha-1)z}\left(\frac{2\sqrt{\kappa}}{1-e^{-2\sqrt{\kappa}z}}\right)^{2-\alpha}.$$
Thus,
\begin{align*}
\mathbb{P}^{0}[X^{(\kappa)}_{T_{]a,\infty[}}\in \mathrm{d}x]/\mathrm{d}x=&\int\limits_{z\in
]x-a,x[}\frac{(2\sqrt{\kappa})^{\alpha}}{(1-e^{-2\sqrt{\kappa}(x-z)})^{\alpha}}\\
&\times\frac{1-\alpha}{\pi}\sin(\alpha\pi)e^{2\sqrt{\kappa}z(\alpha-1)}\frac{(2\sqrt{\kappa})^{2-\alpha}}{(1-e^{-2\sqrt{\kappa}z})^{2-\alpha}}\,\mathrm{d}z.
\end{align*}
By performing the change of variable $t=\frac{1-e^{-2\sqrt{\kappa}z}}{1-e^{-2\sqrt{\kappa}x}}$, we see that
\begin{align*}
\mathbb{P}^{0}[X^{(\kappa)}_{T_{]a,\infty[}}\in \mathrm{d}x]/\mathrm{d}x=&\frac{2\sqrt{\kappa}}{\pi}(1-\alpha)\sin(\alpha\pi)(1-e^{-2\sqrt{\kappa}x})^{-1}\\
&\times\int\limits_{\frac{1-e^{-2\sqrt{\kappa}(x-a)}}{1-e^{-2\sqrt{\kappa}x}}}^{1}(t^{-1}-1)^{-\alpha}t^{-2}\,\mathrm{d}t\\
=&\frac{\sqrt{\kappa}}{\pi}\sin(\alpha\pi)\frac{e^{\alpha\sqrt{\kappa}x}(\sinh(\sqrt{\kappa}a))^{1-\alpha}}{\sinh(\sqrt{\kappa}x)(\sinh(\sqrt{\kappa}(x-a)))^{1-\alpha}}.\qedhere
\end{align*}
\end{proof}

\begin{lem}\label{lem:YTa}
Let $Y^{(\kappa)}$ and $\mathbb{Q}^{0}$ be the same as in Lemma \ref{lem: subordinator bridge}. Fix a positive measurable function $f:[0,1]^2\rightarrow\mathbb{R}_{+}$ and $0<a,b<1$ such that $a+b<1$. Then,
\begin{multline*}
\mathbb{Q}^{0}\left[f(Y^{(\kappa)}_{T_{]a,\infty[}},1-Y^{(\kappa)}_{T_{]1-b,\infty[}-})1_{\{Y^{(\kappa)}_{T_{]a,\infty[}}<1-b\}}\right]\\
=\int\limits_{\mathclap{a<x<1-y<1-b<1}}f(x,y)\frac{\kappa}{\pi^2}\frac{\sin^{2}(\alpha\pi)(\sinh(\sqrt{\kappa}))^{\alpha}}{(\sinh(\sqrt{\kappa}(1-x-y)))^{\alpha}\sinh(\sqrt{\kappa}x)\sinh(\sqrt{\kappa}y)}\\
\times \left(\frac{\sinh(\sqrt{\kappa}a)\sinh(\sqrt{\kappa}b)}{\sinh(\sqrt{\kappa}(x-a))\sinh(\sqrt{\kappa}(y-b))}\right)^{1-\alpha}\,\mathrm{d}x\,\mathrm{d}y.
\end{multline*}
\end{lem}
\begin{proof}
Let $X^{(\kappa)}$ be the subordinator with the potential density $$U(x,y)=1_{\{y>x\}}\left(\frac{2\sqrt{\kappa}}{1-e^{2\sqrt{\kappa}(y-x)}}\right)^{\alpha}$$
and $\mathbb{P}^{0}$ its law starting from $0$. By Lemma \ref{lem:distribution of XTa}, $\mathbb{P}^{0}[X^{(\kappa)}_{T_{]a,\infty[}}=a]=0$ for $a>0$. Next, according to Lemma 1.10 in \cite{BertoinMR1746300}, for $0\leq x<a<x+y$,
$$\mathbb{P}^{0}[X^{(\kappa)}_{T_{]a,\infty[}-}\in \mathrm{d}x, X^{(\kappa)}_{T_{]a,\infty[}}-X^{(\kappa)}_{T_{]a,\infty[}-}\in \mathrm{d}y]=u(x)\mathrm{d}x\Pi(\mathrm{d}y).$$
By applying the strong Markov property at time $T_{]a,\infty[}$ for $X^{(\kappa)}$, we see that
\begin{multline*}
\mathbb{P}^{0}\left[\phi(X^{(\kappa)}_{T_{]a,\infty[}-},X^{(\kappa)}_{T_{]a,\infty[}},X^{(\kappa)}_{T_{]1-b,\infty[}-},X^{(\kappa)}_{T_{]1-b,\infty[}})1_{\{X^{(\kappa)}_{T_{]a,\infty[}}<1-b\}}\right]\\
=\int\limits_{\mathclap{\substack{0<z_1<a<z_1+z_2\\0<z_3<1-b-z_1-z_2<z_3+z_4}}}\phi(z_1,z_1+z_2,z_1+z_2+z_3,z_1+z_2+z_3+z_4)u(z_1)\mathrm{d}z_1\Pi(\mathrm{d}z_2)u(z_3)\mathrm{d}z_3\Pi(\mathrm{d}z_4),
\end{multline*}
where $\phi$ is a positive measurable function. Therefore, for a positive measurable function $\phi$, we have
\begin{multline*}
\mathbb{Q}^{0}\left[\phi(X^{(\kappa)}_{T_{]a,\infty[}-},X^{(\kappa)}_{T_{]a,\infty[}},X^{(\kappa)}_{T_{]1-b,\infty[}-},X^{(\kappa)}_{T_{]1-b,\infty[}})1_{\{X^{(\kappa)}_{T_{]a,\infty[}}<1-b\}}\right]\\
=\mathbb{P}^{0}\left[\phi(X^{(\kappa)}_{T_{]a,\infty[}-},X^{(\kappa)}_{T_{]a,\infty[}},X^{(\kappa)}_{T_{]1-b,\infty[}-},X^{(\kappa)}_{T_{]1-b,\infty[}})\frac{u(1-X^{(\kappa)}_{T_{]1-b,\infty[}})}{u(1)}1_{\{X^{(\kappa)}_{T_{]1-b,\infty[}}<1,X^{(\kappa)}_{T_{]a,\infty[}}<1-b\}}\right]\\
=\int\limits_{\mathclap{0<z_1<a<z_1+z_2<z_1+z_2+z_3<1-b<z_1+z_2+z_3+z_4<1}}\phi(z_1,z_1+z_2,z_1+z_2+z_3,z_1+z_2+z_3+z_4)\\
\times\frac{u(1-z_1-z_2-z_3-z_4)}{u(1)}u(z_1)\mathrm{d}z_1\Pi(\mathrm{d}z_2)u(z_3)\mathrm{d}z_3\Pi(\mathrm{d}z_4).
\end{multline*}
Particularly,
\begin{multline*}
\mathbb{Q}^{0}\left[f(Y^{(\kappa)}_{T_{]a,\infty[}},1-Y^{(\kappa)}_{T_{]1-b,\infty[}-})1_{\{Y^{(\kappa)}_{T_{]a,\infty[}}<1-b\}}\right]\\
=\int\limits_{\mathclap{0<z_1<a<z_1+z_2<z_1+z_2+z_3<1-b<z_1+z_2+z_3+z_4<1}}f(z_1+z_2,1-z_1-z_2-z_3)\\
\times\frac{u(1-z_1-z_2-z_3-z_4)}{u(1)}u(z_1)\mathrm{d}z_1\Pi(\mathrm{d}z_2)u(z_3)\mathrm{d}z_3\Pi(\mathrm{d}z_4).
\end{multline*}
By performing the change of variables $x=z_1+z_2,y=1-z_1-z_2-z_3$, we have that
\begin{multline*}
\mathbb{Q}^{0}\left[f(Y^{(\kappa)}_{T_{]a,\infty[}},1-Y^{(\kappa)}_{T_{]1-b,\infty[}-})1_{\{Y^{(\kappa)}_{T_{]a,\infty[}}<1-b\}}\right]\\
=\int\limits_{\mathclap{0<x-z_2<a<x<1-y<1-b<1-y+z_4<1}}f(x,y)\frac{u(1-x-y)}{u(1)}\mathrm{d}x\mathrm{d}y\cdot u(y-z_4)u(x-z_2)\Pi(\mathrm{d}z_2)\Pi(\mathrm{d}z_4)\\
=\int\limits_{\mathclap{a<x<1-y<1-b<1}}f(x,y)\frac{u(1-x-y)}{u(1)}\mathrm{d}x\mathrm{d}y\\
\times\int\limits_{\mathclap{z_2\in]x-a,x[,z_4\in]y-b,y[}}u(y-z_4)u(x-z_2)\Pi(\mathrm{d}z_2)\Pi(\mathrm{d}z_4).
\end{multline*}
Finally, the calculation is finished by using Lemma \ref{lem:distribution of XTa}.
\end{proof}

\begin{proof}[Proof of Lemma \ref{lem: the cluster at 0}]
Firstly, we deduce from Lemma \ref{lem: limit of the loops passing 1} the density $\rho(a,b)$ of $A,B$:
$$\rho(a,b)=\kappa\alpha(\alpha+1)\left(\frac{2\cosh(\sqrt{\kappa})}{\sinh(\sqrt{\kappa})}\right)^{\alpha}\frac{(\sinh(\sqrt{\kappa a})\sinh(\sqrt{\kappa}b))^{\alpha}}{(\sinh(\sqrt{\kappa}(a+b)))^{\alpha+2}}.$$
Take an independent subordinator bridge $Y^{(\kappa)}$ defined in Lemma \ref{lem: subordinator bridge}. Set
$$(g,d)=(Y^{(\kappa)}_{T_{]B,\infty[}},1-Y^{(\kappa)}_{T_{]1-A,\infty[}-}).$$
By Lemma \ref{lem: limit of the loops passing 1} of the loop clusters formed by $\mathcal{DL}_{\alpha,3}^{(n)}$, a combination of Proposition \ref{conditioned renewal process} and Proposition \ref{convergence of conditioned renewal processes} about the loop clusters formed by $\mathcal{DL}_{\alpha,1}^{(n)}$, the independence between $\mathcal{DL}^{(n)}_{\alpha,1}$ and $\mathcal{DL}^{(n)}_{\alpha,3}$ and the last statement in Lemma \ref{lem: subordinator bridge},
$$\lim\limits_{n\rightarrow\infty}\mathbb{P}[\text{All the edges are not covered by the loops in }\mathcal{DL}^{(n)}_{\alpha,1}\cup\mathcal{DL}^{(n)}_{\alpha,3}]=\mathbb{P}[g+d<1].$$
Moreover, conditionally on the existence of closed edges, $(G_n/n,D_n/n)$ converges in distribution towards $(G,D)$, whose density equals to
$$1_{\{x>0,y>0,x+y<1\}}\frac{\mathbb{P}[g\in \mathrm{d}x,d\in \mathrm{d}y]/\mathrm{d}x\mathrm{d}y}{\mathbb{P}[g+d<1]}.$$
By Lemma \ref{lem:YTa}, for $x>0,y>0,x+y<1$,
\begin{multline*}
\mathbb{P}[g\in \mathrm{d}x,d\in \mathrm{d}y|A=a,B=b]\\
=\mathrm{d}x\mathrm{d}y\cdot1_{\{a<x<1-y<1-b\}}\frac{\kappa}{\pi^2}\frac{\sin^{2}(\alpha\pi)(\sinh(\sqrt{\kappa}))^{\alpha}}{(\sinh(\sqrt{\kappa}(1-x-y)))^{\alpha}\sinh(\sqrt{\kappa}x)\sinh(\sqrt{\kappa}y)}\\
\times \left(\frac{\sinh(\sqrt{\kappa}a)\sinh(\sqrt{\kappa}b)}{\sinh(\sqrt{\kappa}(x-a))\sinh(\sqrt{\kappa}(y-b))}\right)^{1-\alpha}.
\end{multline*}
Therefore,
\begin{align*}
\mathbb{P}[g\in \mathrm{d}x,d\in \mathrm{d}y]=&\int\limits_{0<a<x,0<b<y}\rho(a,b)\mathbb{P}[g\in \mathrm{d}x,d\in \mathrm{d}y|A=a,B=b]\,\mathrm{d}a\,\mathrm{d}b\\
=&\mathrm{d}x\mathrm{d}y\cdot\int\limits_{0<a<x,0<b<y}\frac{\kappa^2}{\pi^2}\frac{\alpha(\alpha+1)\sin^2(\alpha\pi)(2\cosh\sqrt{\kappa})^{\alpha}}{(\sinh(\sqrt{\kappa}(1-x-y)))^{\alpha}\sinh(\sqrt{\kappa}x)\sinh(\sqrt{\kappa}y)}\\
&\times\frac{\sinh(\sqrt{\kappa}a)\sinh(\sqrt{\kappa}b)\,\mathrm{d}a\,\mathrm{d}b}{(\sinh(\sqrt{\kappa}(a+b)))^{\alpha+2}(\sinh(\sqrt{\kappa}(x-a))\sinh(\sqrt{\kappa}(y-b)))^{1-\alpha}}.
\end{align*}
We make a change of variable as follows: $$p=\frac{(1-e^{-2\sqrt{\kappa}a})(1-e^{-2\sqrt{\kappa}x})}{e^{-2\sqrt{\kappa}a}-e^{-2\sqrt{\kappa}x}}\text{ and }q=\frac{(1-e^{-2\sqrt{\kappa}b})(1-e^{-2\sqrt{\kappa}y})}{e^{-2\sqrt{\kappa}b}-e^{-2\sqrt{\kappa}y}}.$$
Accordingly,
\begin{align*}
\mathbb{P}[g\in \mathrm{d}x,d\in \mathrm{d}y]=&\mathrm{d}x\mathrm{d}y\cdot\frac{1}{\pi^2}\frac{2^{2\alpha-2}\kappa\alpha(\alpha+1)\sin^2(\alpha\pi)(\cosh\sqrt{\kappa})^{\alpha}}{(\sinh(\sqrt{\kappa}(1-x-y)))^{\alpha}(\sinh(\sqrt{\kappa}x)\sinh(\sqrt{\kappa}y))^{2-\alpha}}\\
&\times\int\limits_{p,q>0}\frac{pq\cdot \mathrm{d}p\mathrm{d}q}{\left(\frac{1-e^{-2\sqrt{\kappa}(x+y)}}{(1-e^{-2\sqrt{\kappa}x})(1-e^{-2\sqrt{\kappa}y})}pq+p+q\right)^{\alpha+2}}.
\end{align*}
For the simplicity of notation, set $\delta=\frac{1-e^{-2\sqrt{\kappa}(x+y)}}{(1-e^{-2\sqrt{\kappa}x})(1-e^{-2\sqrt{\kappa}y})}$. By performing the change of variable $z=\frac{p}{\delta pq+p+q}$,
\begin{align*}
\int\limits_{p,q>0}\frac{pq\cdot \mathrm{d}p\mathrm{d}q}{\left(\frac{1-e^{-2\sqrt{\kappa}(x+y)}}{(1-e^{-2\sqrt{\kappa}x})(1-e^{-2\sqrt{\kappa}y})}pq+p+q\right)^{\alpha+2}}=&\int\limits_{p,q>0}\frac{pq\,\mathrm{d}p\,\mathrm{d}q}{(p+q+\delta pq)^{\alpha+2}}\\
=&\int\limits_{p>0,z\in[0,1]}\frac{p^{1-\alpha}}{(1+\delta p)^2}z^{\alpha-1}(1-z)\,\mathrm{d}p\,\mathrm{d}z\\
=&\frac{1}{\alpha(\alpha+1)}\int\limits_{0}^{\infty}\frac{p^{1-\alpha}}{(1+\delta p)^2}\,\mathrm{d}p.
\end{align*}
We take $w=\frac{1}{1+\delta p}$:
\begin{align*}
\frac{1}{\alpha(\alpha+1)}\int\limits_{0}^{\infty}\frac{p^{1-\alpha}}{(1+\delta p)^2}\,\mathrm{d}p=&\frac{1}{\alpha(\alpha+1)\delta^{2-\alpha}}\int\limits_{0}^{1}w^{\alpha-1}(1-w)^{1-\alpha}\,dw\\
=&\frac{1}{\alpha(\alpha+1)\delta^{2-\alpha}}\text{Beta}(2-\alpha,\alpha)\\
=&\frac{1-\alpha}{\alpha(\alpha+1)\delta^{2-\alpha}}\text{Beta}(1-\alpha,\alpha).
\end{align*}
By Euler's reflection formula, $\text{Beta}(1-\alpha,\alpha)=\frac{\pi}{\sin(\pi\alpha)}$. Thus, for $x>0,y>0,x+y<1$,
$$\mathbb{P}[g\in \mathrm{d}x,d\in \mathrm{d}y]=\mathrm{d}x\mathrm{d}y\cdot\frac{\sin(\alpha\pi)}{\pi}\frac{2^{\alpha}(1-\alpha)\kappa(\cosh\sqrt{\kappa})^{\alpha}}{\left[\sinh(\sqrt{\kappa}(1-x-y))\right]^{\alpha}\left[\sinh(\sqrt{\kappa}(x+y))\right]^{2-\alpha}}.$$

Denote by $Pr$ the quantity $\int\limits_{x>0,y>0,x+y<1}\mathbb{P}[g\in \mathrm{d}x,d\in \mathrm{d}y]$. Then,
\begin{align*}
Pr=&\int\limits_{x>0,y>0,x+y<1}\frac{\sin(\alpha\pi)}{\pi}\frac{\kappa(1-\alpha)(2\cosh\sqrt{\kappa})^{\alpha}\,\mathrm{d}x\,\mathrm{d}y}{[\sinh(\sqrt{\kappa}(1-x-y))]^{\alpha}[\sinh(\sqrt{\kappa}(x+y))]^{2-\alpha}}\\
=&\int\limits_{0<x<z<1}\frac{\sin(\alpha\pi)}{\pi}\frac{\kappa(1-\alpha)(2\cosh\sqrt{\kappa})^{\alpha}\,\mathrm{d}x\,\mathrm{d}z}{[\sinh(\sqrt{\kappa}(1-z))]^{\alpha}[\sinh(\sqrt{\kappa}z)]^{2-\alpha}}\\
=&\int\limits_{0}^{1}\frac{\sin(\alpha\pi)}{\pi}\frac{\kappa(1-\alpha)(2\cosh\sqrt{\kappa})^{\alpha}z\,\mathrm{d}z}{[\sinh(\sqrt{\kappa}(1-z))]^{\alpha}[\sinh(\sqrt{\kappa}z)]^{2-\alpha}}.
\end{align*}
Take $s=\frac{1-e^{-2\sqrt{\kappa}z}}{1-e^{-2\sqrt{\kappa}}}$:
\begin{align*}
Pr=&\frac{\sin(\alpha\pi)}{\pi}(1-\alpha)(2\cosh\sqrt{\kappa})^{\alpha}\frac{e^{-\alpha\sqrt{\kappa}}}{1-e^{-2\sqrt{\kappa}}}\\
&\times\int\limits_{0}^{1}s^{\alpha-2}(1-s)^{-\alpha}(-\log(1-(1-e^{-2\sqrt{\kappa}})s))\,\mathrm{d}s.
\end{align*}
By Taylor expansion, $-\log(1-(1-e^{-2\sqrt{\kappa}})s)=\sum\limits_{n=1}^{\infty}\frac{1}{n}(1-e^{-2\sqrt{\kappa}})^ns^n$, hence
\begin{align*}
Pr=&\frac{\sin(\alpha\pi)}{\pi}(1-\alpha)(2\cosh\sqrt{\kappa})^{\alpha}\frac{e^{-\alpha\sqrt{\kappa}}}{1-e^{-2\sqrt{\kappa}}}\sum\limits_{n\geq 1}\int\limits_{0}^{1}\frac{(1-e^{-2\sqrt{\kappa}})^n}{n}s^{\alpha-2+n}(1-s)^{-\alpha}\,\mathrm{d}s\\
=&\frac{\sin(\alpha\pi)}{\pi}(1-\alpha)(2\cosh\sqrt{\kappa})^{\alpha}\frac{e^{-\alpha\sqrt{\kappa}}}{1-e^{-2\sqrt{\kappa}}}\sum\limits_{n\geq 1}\frac{(1-e^{-2\sqrt{\kappa}})^n}{n}\text{Beta}(\alpha+n-1,1-\alpha)\\
=&\frac{\sin(\alpha\pi)}{\pi}(1-\alpha)\Gamma(1-\alpha)(2\cosh\sqrt{\kappa})^{\alpha}\frac{e^{-\alpha\sqrt{\kappa}}}{1-e^{-2\sqrt{\kappa}}}\sum\limits_{n\geq 1}\frac{(1-e^{-2\sqrt{\kappa}})^{n}}{n!}\Gamma(\alpha+n-1).
\end{align*}
We have
\begin{align*}
\sum\limits_{n\geq 1}\frac{(1-e^{-2\sqrt{\kappa}})^{n}}{n!}\Gamma(\alpha+n-1)=&\sum\limits_{n\geq 1}\frac{(1-e^{-2\sqrt{\kappa}})^{n}}{n!}\int\limits_{0}^{\infty}e^{-t}t^{\alpha-2+n}\,\mathrm{d}t\\
=&\int\limits_{0}^{\infty}e^{-t}t^{\alpha-2}(\sum\limits_{n\geq 1}\frac{(1-e^{-2\sqrt{\kappa}})^{n}t^{n}}{n!})\,\mathrm{d}t\\
=&\int\limits_{0}^{\infty}(e^{-e^{-2\sqrt{\kappa}}t}-e^{-t})t^{\alpha-2}\,\mathrm{d}t.
\end{align*}
By integration by parts,
\begin{align*}
\sum\limits_{n\geq 1}\frac{(1-e^{-2\sqrt{\kappa}})^{n}}{n!}\Gamma(\alpha+n-1)=&\left.(e^{-e^{-2\sqrt{\kappa}}t}-e^{-t})\frac{t^{\alpha-1}}{\alpha-1}\right|^{\infty}_{0}\\
&-\frac{1}{\alpha-1}\int\limits_{0}^{\infty}(e^{-t}-e^{-2\sqrt{\kappa}}e^{-e^{-2\sqrt{\kappa}}t})t^{\alpha-1}\,\mathrm{d}t\\
=&\frac{\Gamma(\alpha)(1-e^{-2\sqrt{\kappa}(1-\alpha)})}{1-\alpha}.
\end{align*}
Hence,
$$\mathbb{P}[g+d<1]=\frac{(2\cosh\sqrt{\kappa})^{\alpha}\sinh(\sqrt{\kappa}(1-\alpha))}{\sinh\sqrt{\kappa}}.$$
Finally, one can deduce the distribution of $(G,D)$.
\end{proof}

\section{Informal relation with convergence of loop-soups}
In this section, we would like to give informal remarks of the previous results from the point of view of the scaling limit of the loop-soup. Please refer to \cite{titus} for the Markovian loop-soup of one dimensional diffusions.

Firstly, let us give an informal explanation of the convergence result for the closed edges in the loop cluster model on $\mathbb{N}$ which is proved in \cite{Markovian-loop-clusters-on-graphs}.

It is known that the Brownian loop-soup is the scaling limit of simple random walk loop-soup. Intuitively, the scaling limit of the closed edges probably\footnote{There is not an immediate consequence of the convergence of loop-soup. That's why our explanation stays informal.} has some relation with the zero set of the occupation field of the Brownian loop. As an application of \cite[Proposition 4.5]{titus}, the occupation field of Brownian loop-soup with killing rate $\frac{\kappa}{2}$ within $]0,\infty[$ is a homogeneous branching process with immigration. It is the solution of the following SDE:
$$\mathrm{d}X_t=2\sqrt{X_t}\mathrm{d}B_t-2\sqrt{\kappa}X_t\,\mathrm{d}t+2\alpha\,\mathrm{d}t,t\in[0,\infty[,$$
where $B$ is a Brownian motion and $X_{0}=0$. (It belongs to the Cox-Ingersoll-Ross (CIR) family of diffusions which could be viewed as a generalization of squared Bessel process. More precisely, it is a radial Ornstein-Uhlenbeck process of dimension $2\alpha$ with parameter $-\sqrt{\kappa}$, see \cite{AnjaMR1997032}.) To be more precise, when we apply  Proposition 4.5 in \cite{titus}, we take the non-increasing positive harmonic function to be $u_{\downarrow}(x)=e^{-\sqrt{\kappa}x}$ and take the non-decreasing positive harmonic function to be $u_{\uparrow}(x)=\frac{2}{\sqrt{\kappa}}\sinh(\sqrt{\kappa}x)$ such that the Green function density with respect to the Lebesgue measure $G(x,y)$ is given by $G(x,y)=u_{\uparrow} (x)u_{\downarrow} (y)$ for $x\leq y$. (This normalization is required when applying Proposition 4.5 in \cite{titus}). We see that $w(x)=\operatorname{Wronskian}(u_{\downarrow},u_{\uparrow})=2$. One can check that the zero set is given by the range of the subordinator with potential density 
$U(x,y)=1_{\{y>x\}}\left(\frac{2\sqrt{\kappa}}{1-e^{-2\sqrt{\kappa}(y-x)}}\right)^{\alpha}$, see e.g. \cite[Proposition 2.2]{BertoinMR1746300}.

Next, we consider the loop cluster over a discrete interval which is considered in this article. If the approximation by Brownian loop-soup within $]0,1[$ works, then we expect that the limit distribution of the closed edges is the zero set of the occupation field of this Brownian loop-soup. By Proposition 4.5 of \cite{titus}, we know that the occupation field over the interval $]0,1[$ indexed by the position $t\in]0,1[$ is the solution of the following SDE:
$$\mathrm{d}Y_t=2\sqrt{Y_t}\mathrm{d}B_t+\left(2\alpha-\frac{\cosh(\sqrt{\kappa}(1-t))}{\sinh(\sqrt{\kappa}(1-t))}Y_t\right)\,\mathrm{d}t,t\in[0,1].$$
In fact, it is the bridge of a squared radial OU process of dimension $2\alpha$ of parameter $-\sqrt{\kappa}$ from $0$ to $0$ of fixed time duration $1$. Please refer to \cite{markovianbridge} for Markovian bridge and refer to \cite{AnjaMR1997032} for the transition density of squared radial OU process and its relationship with squared Bessel process. Let $D_t$ be the first time of hitting $0$ after time $t$. Then, the Radon-Nikodym derivative of the bridge process over the squared radial OU process is
$$1_{\{D_t<1\}}\left(\frac{1-e^{-\sqrt{\kappa}}}{1-e^{-2\sqrt{\kappa}(1-D_t)}}\right)^{\alpha},$$
restricted on the sub-$\sigma$-field up to time $D_t$. This is exactly the same as $\frac{U(D_t,1)}{U(0,1)}$, which is used to construct our subordinator bridge. Then, one can check that the zero set of the bridge of the squared radial OU process agrees with the range of the conditioned subordinator defined in Lemma \ref{lem: subordinator bridge}.

Finally, we would like to point out the way to get the limit distribution of $(A,B)$ in Lemma \ref{lem: limit of the loops passing 1} from the point of view of Brownian loops. By the structure of Poisson random measure, it is enough to check this for $\alpha=1$. In this case, there is a connection between the loops passing through a fixed point and the Poisson point process of excursions at the same point, see e.g. \cite{loop}, \cite{titus}. For $\alpha=1$, they agree with each other. Accordingly, the distribution of $[-A,B]$ is exactly the random interval covered by these excursions under the condition that they don't cover $-1$ nor $1$. The condition of avoiding $-1$ and $1$ only affects the joint density of $(A,B)$ up to a normalization constant. Therefore, we could remove this restriction for the moment. The total local time at $0$ is an exponential variable with expectation $G(x,x)=1/\sqrt{\kappa}$ since the excursions at $0$ form a Poisson point process. The occupation time (total local time) indexed by the position $x\in]-\infty,\infty[$ forms a two-sided process, the part on the left hand side of $0$ is denoted by $(U_{-x},x\geq 0)$ under time reversal and the right part is denoted by $(V_x,x\geq 0)$. By Ray-Knight theorem for diffusions, conditioned on the total local time at $0$, $U$ and $V$ are two independent copies of squared radial OU processes of dimension zero and parameter $-\sqrt{\kappa}$, see e.g. Proposition 4.1 \cite{titus}. Thus, it is enough to compute the first hitting time of $0$, and then integrate them with respect to the total local time. The density of the first hitting time of $0$ for our squared radial OU process is given by
$$t\rightarrow \frac{x^2}{2}\left(\frac{\sqrt{\kappa}}{\sinh(\sqrt{\kappa t})}\right)^{2}\exp\{\frac{\sqrt{\kappa}}{2}x^2(1-\coth(\sqrt{\kappa}t))\},$$
see e.g. \cite{ElworthyMR1725406} (Corollary 3.19).
Finally, we get the joint density of the first hitting times of $0$ for $U$ and $V$. We see that it is exactly the same density as the limit distribution of $(A_n/n,B_n/n)$ as $n\rightarrow\infty$ up to a normalization constant, see Lemma \ref{lem: limit of the loops passing 1}.

\section{Appendix}\label{sect: appendix}
\subsection{Proof of Lemma \ref{lem: subordinator bridge}}
\begin{enumerate}
\item The subordinator $(X^{(\kappa)}_t,t\geq 0)$ has the potential density $U(x,y)=\left(\frac{2\sqrt{\kappa}}{1-e^{-2\sqrt{\kappa}(y-x)}}\right)^{\alpha}$ for $y>x$. When $y$ tends to $x$, $U(x,y)$ tends to $\infty$. As a consequence, the drift coefficient $d=0$, see Proposition 1.7 in \cite{BertoinMR1746300}. It is proved by Kesten \cite{KestenMR0272059} that for a fixed $x>0$, $x$ does not belong to the range of the subordinator with probability $1$, see Proposition 1.9 in \cite{BertoinMR1746300}. By applying the strong Markov property at a stopping time $S$,
\begin{multline*}
\mathbb{E}^{0}[f(X^{\kappa}_s,s\in[0,S])1_{\{S<T_{]1,+\infty[}\}},X^{(\kappa)}_{T_{]1,+\infty[-}}\in \mathrm{d}b]\\
=\mathbb{E}^{0}\left[f(X^{\kappa}_s,s\in[0,S])1_{\{S<T_{]1,+\infty[}\}}\mathbb{E}^{X^{(\kappa)}_S}[X^{(\kappa)}_{T_{]1,+\infty[-}}\in \mathrm{d}b]\right].
\end{multline*}
By \cite[Lemma 1.10]{BertoinMR1746300}, we get that
\begin{itemize}
\item $\mathbb{E}^{X^{(\kappa)}_S}[X^{(\kappa)}_{T_{]1,+\infty[-}}\in \mathrm{d}b]=\bar{\Pi}(1-b)u(b-X^{(\kappa)}_S)\,\mathrm{d}b$,\footnote{Here, $\bar{\Pi}$ represents the tail of the L\'{e}vy measure of the subordinator.}
\item $\mathbb{E}^{0}[X^{(\kappa)}_{T_{]1,+\infty[-}}\in \mathrm{d}b]=\bar{\Pi}(1-b)u(b)\,\mathrm{d}b=\frac{u(b)}{u(b-X^{(\kappa)}_s)}\mathbb{P}^{X_s^{(\kappa)}}[X_{T_{]1,+\infty[}-}^{(\kappa)}\in \mathrm{d}b].$
\end{itemize}
Hence,
\begin{multline*}
\mathbb{E}^{0}\left[f(X^{\kappa}_s,s\in[0,S])1_{\{S<T_{]1,+\infty[}\}}\mathbb{E}^{X^{(\kappa)}_S}[X^{(\kappa)}_{T_{]1,+\infty[-}}\in \mathrm{d}b]\right]\\
=\mathbb{E}^{0}[X^{(\kappa)}_{T_{]1,+\infty[}}\in \mathrm{d}b]\mathbb{E}^{0}\left[f(X^{\kappa}_s,s\in[0,S])1_{\{S<T_{]1,+\infty[}\}}\frac{u(b-X^{(\kappa)}_{S})}{u(b)}\right].
\end{multline*}
In particular, for a fixed time $t$, we have that
\begin{multline*}\label{eq:4}
\mathbb{E}^{0}[f(X^{\kappa}_s,s\in[0,t])1_{\{t<T_{]1,+\infty[}\}}|X^{(\kappa)}_{T_{]1,+\infty[-}}=1]\\
=\mathbb{E}^{0}\left[f(X^{\kappa}_s,s\in[0,t])1_{\{t<T_{]1,+\infty[}\}}\frac{u(1-X^{(\kappa)}_t)}{u(1)}\right].\tag{\text{$*$}}
\end{multline*}
\item It is enough for us to show that $x\rightarrow u(1-x)=U(x,1)$ is excessive. The rest will follow from the classical results on the Doob's $h$-transform, see Chapter 11 of \cite{ChungMR2152573}. Take a positive function $g$,  we have $P^{(\kappa)}_tUg=\int\limits_{t}^{\infty}P^{(\kappa)}_sg\,\mathrm{d}s$ and $Ug=\int\limits_{0}^{\infty}P^{(\kappa)}_sg\,\mathrm{d}s$. Thus, for all positive function $g$, we have $P^{(\kappa)}_tUg\leq Ug$ and $P^{(\kappa)}_tUg$ increases to $Ug$ as $t$ decreases to $0$. As a consequence, except for a set $N$ of $z$ of zero Lebesgue measure, $y\rightarrow u(y,z)$ is an excessive function, i.e.
    \begin{itemize}
    \item $\int P^{(\kappa)}_t(x,\mathrm{d}y)u(y,z)\leq u(x,z)$,
    \item $\lim\limits_{t\rightarrow 0}P^{(\kappa)}_t(x,\mathrm{d}y)u(y,z)=u(x,z)$.
    \end{itemize}
    Take a decreasing sequence $(z_n)_n$ with limit $1$ which is outside of the negligible set $N$. As the increasing limit of a sequence of excessive functions $y\rightarrow u(y,z_n)$, $y\rightarrow u(y,1)$ is excessive.
\item Before the proof, we would like to give a short explanation. From the symmetry of the loop model on the discrete segment, the graphs of the conditional renewal processes are centrosymmetric. Therefore, as the scaling limit, the conditional subordinator has a centrosymmetric graph. Thus, $Y^{(\kappa)}_{\zeta-}=1$ is equivalent to $Y^{(\kappa)}_{0+}=0$ which is obviously true.

    \bigskip
    In the following, we will not use the discrete approximation described above. Instead, we will prove that $\mathbb{Q}^{x}[X_{T_{[1-\delta,\infty[}}\in ]0,1[]=\mathbb{Q}^{x}[X_{T_{[x+\delta,\infty[}}\in ]0,1[]$ which is motivated by the idea of time reversal.

Let's begin to prove $Y^{(\kappa)}_{\zeta-}=1$. To prove this, it is enough to show that
$$\mathbb{Q}^{x}[T_{[1-\delta,\infty[}<\zeta]=1\text{ for all }\delta>0.$$
By applying Theorem 11.9 of \cite{ChungMR2152573} to the stopping time $T_{[1-\delta,\infty[}$, we get that
\begin{equation*}
\mathbb{Q}^{x}[T_{[1-\delta,\infty[}<\zeta]=\mathbb{P}^{x}\left[T_{[1-\delta,\infty[}<T_{]1,+\infty[},\frac{u(1-X_{T_{[1-\delta,\infty[}})}{u(1-x)}\right].
\end{equation*}
If $X$ follows the law $\mathbb{P}^{0}$, then $X+x$ has the law $\mathbb{P}^x$. Therefore, the above quantity equals to
$$\mathbb{P}^{0}\left[T_{[1-x-\delta,\infty[}<T_{]1-x,\infty[},\frac{u(1-x-X_{T_{[1-x-\delta,\infty[}})}{u(1-x)}\right].$$
By Lemma 1.10 in \cite{BertoinMR1746300}, for $0\leq a<1-x-\delta\leq a+b$, we have that
$$\mathbb{P}^{0}[X_{T_{[1-x-\delta,\infty[}-}\in \mathrm{d}a,X_{T_{[1-x-\delta,\infty[}}-X_{T_{[1-x-\delta,\infty[}-}\in \mathrm{d}b]=u(a)\,\mathrm{d}a\Pi(\mathrm{d}b).$$
Consequently,
$$\mathbb{Q}^{x}[T_{[1-\delta,\infty[}<\zeta]=\int\limits_{0<a<1-x-\delta<a+b<1-x} \frac{u(1-x-a-b)}{u(1-x)}u(a)\,\mathrm{d}a\Pi(\mathrm{d}b).$$
By performing the change of variable $c=1-x-a-b$, we see that
\begin{align*}
\mathbb{Q}^{x}[T_{[1-\delta,\infty[}<\zeta]=&\int\limits_{0<c<\delta<c+b<1-x} \frac{u(c)}{u(1-x)}u(1-x-c-b)\,\mathrm{d}c\Pi(\mathrm{d}b)\\
=&\mathbb{P}^{0}\left[T_{[\delta,\infty[}<T_{]1-x,\infty[},\frac{u(1-x-X_{T_{[\delta,\infty[}})}{u(1-x)}\right]\\
=&\mathbb{P}^{x}\left[T_{[x+\delta,\infty[}<T_{]1,+\infty[},\frac{u(1-X_{T_{[x+\delta,\infty[}})}{u(1-x)}\right]\\
=&\mathbb{Q}^{x}[T_{[x+\delta,\infty[}<\zeta].
\end{align*}
By the right-continuity of the path, $\lim\limits_{\delta\rightarrow 0}\mathbb{Q}^{x}[T_{[x+\delta,\infty[}<\zeta]=1$. Hence,
$$\lim\limits_{\delta\rightarrow 0}\mathbb{Q}^{x}[T_{[1-\delta,\infty[}<\zeta]=\lim\limits_{\delta\rightarrow 0}\mathbb{Q}^{x}[T_{[x+\delta,\infty[}<\zeta]=1.$$
Since $a\rightarrow\mathbb{Q}^{x}[T_{[x+a,\infty[}<\zeta]$ is non-increasing, we must have
$$\mathbb{Q}^{x}[T_{[y,\infty[}<\zeta]=1\text{ for }y\in[x,1[.$$
\item We know that $P^{(\kappa)}_t$ is a Feller semi-group. For $f\in C_K([0,1[)$, $x\rightarrow Q^{(\kappa)}_tf(x)$ belongs to $C_{K}([0,1[)$. ($C_K([0,1[)$ denotes the collection of compact supported continuous functions over $[0,1[$ and
    $$C_0([0,1[)=\{f:[0,1[\rightarrow\mathbb{R}: f\text{ is continuous and }\lim\limits_{x\rightarrow 1}f(y)=0\}.\text{)}$$
    By the Markov property of the semi-group $(Q^{(\kappa)}_t)_{t\geq 0}$, $||Q^{(\kappa)}_tf||_{\infty}\leq ||f||_{\infty}$ for every $f\in C_{0}([0,1])$. Thus, we have
    $Q^{(\kappa)}_tf=\lim\limits_{n\rightarrow\infty}Q^{(\kappa)}_t(f|_{[0,1-1/n[})\in C_0([0,1[)$. For $x\in[0,1[$ and $f\in C_0([0,1[)$, $$\lim\limits_{t\rightarrow 0}Q^{(\kappa)}_tf(x)=\lim\limits_{t\rightarrow 0}\mathbb{P}^{x}\left[1_{\{t<T_{]1,+\infty[}\}}f(X^{(\kappa)}_t)\frac{u(1-X^{(\kappa)}_t)}{u(1-x)}\right]\underset{ \text{convergence}}{\overset{\text{dominated}}{=}}f(x).$$
    In other words, the semi-group $(Q^{(\kappa)}_t)_{t\geq 0}$ is Feller.
\item By a classical result about time reversal, the reversed process is a moderate Markov process, its semi-group $\hat{Q}^{(\kappa)}_t(x,\mathrm{d}y)$ is given by the following formula:
    $$\langle g,Q^{(\kappa)}_tf\rangle_{G}=\langle \hat{Q}^{(\kappa)}_tg,f\rangle_{G},$$
    where $Q^{(\kappa)}_t(x,\mathrm{d}y)=\frac{U(y,1)}{U(x,1)}P^{(\kappa)}_t(x,\mathrm{d}y)$ and $G(\mathrm{d}x)=\int\limits_{0}^{\infty}Q^{(\kappa)}_t(0,\mathrm{d}x)\,\mathrm{d}t=\frac{U(0,x)U(x,1)}{U(0,1)}\,\mathrm{d}x$. Denote by $(\hat{P}^{(\kappa)}_t)_{t\geq 0}$ the dual semi-group of $(P^{(\kappa)}_t)_{t\geq 0}$ (or the semi-group of $-X^{(\kappa)}$ equivalently). Denote by $u(x)$ the function $U(0,x)$ and by $h(x)$ the function $U(x,1)$. Then,
    $$\langle g,Q^{(\kappa)}_tf\rangle_{G}=\int\limits_{0}^{1}\frac{P^{(\kappa)}_t(hf)(x)}{h(x)}g(x)\frac{u(x)h(x)}{u(1)}\,\mathrm{d}x.$$
    Then we use the duality between $(P^{(\kappa)}_t)_{t\geq 0}$ and $(\hat{P}^{(\kappa)}_t)_{t\geq 0}$:
    \begin{align*}
    \langle g,Q^{(\kappa)}_tf\rangle_{G}=&\int\limits_{0}^{1} f(x)\frac{\hat{P}^{(\kappa)}_t(ug)}{u(x)}\frac{u(x)h(x)}{u(1)}\,\mathrm{d}x\\
    =&\left\langle f,\frac{\hat{P}^{(\kappa)}_t(ug)}{u}\right\rangle_G.
    \end{align*}
    This implies that the semi-group $(\hat{Q}^{(\kappa)}_t)_{t\geq 0}$ associated with the reversed process of $Y$ is given by
    $$\hat{Q}^{(\kappa)}_t(x,\mathrm{d}y)=\hat{P}^{(\kappa)}_t(x,\mathrm{d}y)\frac{U(0,y)}{U(0,x)}=\hat{Q}^{(\kappa)}_t(x,\mathrm{d}y)\frac{U(1-y,1)}{U(1-x,1)}.$$
    By a change of variable, we find that it equals to the semi-group of $1-Y^{(\kappa)}$. By result 3 in this lemma, the reversed process starts from $1$. Then, it is exactly the left-continuous modification of $1-Y^{(\kappa)}$ for $Y^{(\kappa)}$ starting from $0$.
\item Since $U(0,0+)=\infty$, by a result of J. Neveu \cite{NeveuMR0130714} (see \cite[Proposition 1.9 (ii)]{BertoinMR1746300}), the subordinator $X^{(\kappa)}$ has zero drift. Then, by a result of Kesten \cite{KestenMR0272059} (See \cite[Proposition 1.9 (i)]{BertoinMR1746300}), for $x>0$, $\mathbb{P}[x\in\bar{R}(X^{(\kappa)})]=0$ where
$$\bar{R}(X^{(\kappa)})\overset{\mathrm{def}}{=}\text{closure of the range of the subordinator }X^{(\kappa)}.$$
Set $\bar{R}(Y^{(\kappa)})\overset{\mathrm{def}}{=}\text{closure of the range of the conditioned subordinator }Y^{(\kappa)}$. Finally, by the second and the third result of this lemma, for any $\delta>0$, the distributions of $\bar{R}(X^{(\kappa)})\cap[0,1-\delta]$ and $\bar{R}(Y^{(\kappa)})\cap[0,1-\delta]$ are absolute continuous to each other. Hence, for $0<x<1$, $\mathbb{Q}^{0}[x\in \bar{R}(Y^{(\kappa)})]=0$.
\end{enumerate}

\subsection{Proof of Lemma \ref{lem: tightness and skorokhod convergence towards a subordinator}}
Set $\tilde{\mathcal{G}}_{m}^{(n)}=\sigma(S^{(\kappa^{(n)})}_1,\ldots,S^{(\kappa^{(n)})}_m)$ for $m\geq 0$ and $\mathcal{G}^{(n)}_t=\tilde{\mathcal{G}}_{\lfloor n^{1-\alpha}t\rfloor}$ for $t\geq 0$. By definition, $(\mathcal{G}^{(n)}_t)_{t\geq 0}$ is a right-continuous filtration. As usual, by adding the negligible sets, we get a complete filtration which is denoted by the same notation. When $T$ is a $(\mathcal{G}^{(n)}_t)_{t\geq 0}$-stopping time, $\lfloor n^{1-\alpha}T\rfloor$ is a $(\tilde{\mathcal{G}}_{m}^{(n)})_{n\geq 0}$ stopping time. For the tightness, it is enough to verify the following Aldous' criteria (see \cite{JacodShiryaevMR1943877}): for each strictly positive $M$ and $\delta$,
\begin{align}
&\lim\limits_{K\rightarrow\infty}\lim\limits_{n\rightarrow\infty}\mathbb{P}\left[\frac{1}{n} S^{(\kappa^{(n)})}_{\lfloor n^{1-\alpha}M\rfloor}>K\right]=0\label{eq:aldous' criterion a},\\
&\lim\limits_{\theta\downarrow 0}\lim\limits_{n\rightarrow\infty}\sup\limits_{\substack{T_1,T_2\in\mathcal{T}^{(n)}_M,\\T_1\leq T_2\leq T_1+\theta}}\mathbb{P}\left[\left|\frac{1}{n} S^{(\kappa^{(n)})}_{\lfloor n^{1-\alpha} T_2\rfloor}-\frac{1}{n} S^{(\kappa^{(n)})}_{\lfloor n^{1-\alpha} T_1\rfloor}\right|>\delta\right]=0,\notag
\end{align}
where $\mathcal{T}^{(n)}_M$ is the collection of $(\mathcal{G}^{(n)}_t)_{t\geq 0}$-stopping times bounded by $M$. Condition \eqref{eq:aldous' criterion a} is implied by finite marginals convergence and $\mathbb{P}[X^{(\kappa)}_M=\infty]=0$. Since $S^{(\kappa^{(n)})}$ is a renewal process, for $T_1,T_2\in\mathcal{T}^{(n)}_M$ such that $T_1\leq T_2\leq T_1+\theta$, we have that $$\left|\frac{1}{n} S^{(\kappa^{(n)})}_{\lfloor n^{1-\alpha} T_2\rfloor}-\frac{1}{n} S^{(\kappa^{(n)})}_{\lfloor n^{1-\alpha} T_1\rfloor}\right|\leq \left|\frac{1}{n} S^{(\kappa^{(n)})}_{\lfloor n^{1-\alpha} T_1\rfloor+\lceil n^{1-\alpha}\theta\rceil}-\frac{1}{n} S^{(\kappa^{(n)})}_{\lfloor n^{1-\alpha} T_1\rfloor}\right|\overset{\text{law}}{=}\left|\frac{1}{n} S^{(\kappa^{(n)})}_{\lceil n^{1-\alpha}\theta\rceil}\right|.$$
By finite marginals convergence, we get that
\begin{multline*}
\lim\limits_{\theta\downarrow 0}\lim\limits_{n\rightarrow\infty}\sup\limits_{T_1,T_2\in\mathcal{T}^{(n)}_M,T_1\leq T_2\leq T_1+\theta}\mathbb{P}\left[\left|\frac{1}{n} S^{(\kappa^{(n)})}_{\lfloor n^{1-\alpha} T_2\rfloor}-\frac{1}{n}S^{(\kappa^{(n)})}_{\lfloor n^{1-\alpha}T_1\rfloor}\right|>\delta\right]\\
\leq\lim\limits_{\theta\downarrow 0}\lim\limits_{n\rightarrow\infty}\mathbb{P}\left[\left|\frac{1}{n}S^{(\kappa^{(n)})}_{\lceil n^{1-\alpha}\theta\rceil}\right|>\delta\right]\leq\lim\limits_{\theta\downarrow 0}\mathbb{P}[|X^{(\kappa)}_{2\theta}|>\delta]=0
\end{multline*}
and the proof is complete.

\subsection{Proof of Lemma \ref{lem: levy measure}}
The L\'{e}vy measure $\Pi$ and the renewal density $u(\cdot)=U(0,\cdot)$ are related through the Laplace exponent of the subordinator as follows:
\begin{align*}
\frac{1}{\Phi(\lambda)}&=\int\limits_{0}^{\infty}e^{-\lambda x}u(x)\,\mathrm{d}x,\\
\Phi(\lambda)&=\int\limits_{0}^{\infty}(1-e^{-\lambda x})\,\Pi(\mathrm{d}x)=\lambda\int\limits_{0}^{\infty}e^{-\lambda t}\bar{\Pi}(t),
\end{align*}
where $\bar{\Pi}$ is the tail mass of $\Pi$. We compute $\Phi(\lambda)$ from $u(x)=\left(\frac{2\sqrt{\kappa}}{1-e^{-2\sqrt{\kappa}x}}\right)^{\alpha}$:
\begin{align*}
\frac{1}{\Phi(\lambda)}=&\int\limits_{0}^{\infty}e^{-\lambda x}u(x)\,\mathrm{d}x\\
=&\int\limits_{0}^{\infty}\left(\frac{2\sqrt{\kappa}}{1-e^{-2\sqrt{\kappa}x}}\right)^{\alpha}e^{-\lambda x}\,\mathrm{d}x.
\end{align*}
We change the variable $x$ by $\frac{\log(1-s)}{-2\sqrt{\kappa}}$:
\begin{align*}
\frac{1}{\Phi(\lambda)}=&\int\limits_{0}^{1}\left(\frac{2\sqrt{\kappa}}{s}\right)^{\alpha}e^{-\lambda\frac{\log(1-s)}{-2\sqrt{\kappa}}}\frac{1}{2\sqrt{\kappa}(1-s)}\,\mathrm{d}s\\
=&(2\sqrt{\kappa})^{\alpha-1}\int\limits_{0}^{1}s^{-\alpha}(1-s)^{\frac{\lambda}{2\sqrt{\kappa}}-1}\,\mathrm{d}s\\
=&(2\sqrt{\kappa})^{\alpha-1}\text{Beta}\left(\frac{\lambda}{2\sqrt{\kappa}},1-\alpha\right).
\end{align*}
By applying the following equality\footnote{This is implied by Euler's reflection principle: $\Gamma(\alpha)\Gamma(1-\alpha)=\frac{\pi}{\sin(\alpha\pi)}$.}
$$\text{Beta}(x,y) \cdot \text{Beta}(x+y,1-y) = \frac{\pi}{x \sin(\pi y)},$$
with $x=\frac{\lambda}{2\sqrt{\kappa}}$ and $y=1-\alpha$, we get that
\begin{align*}
\Phi(\lambda)=&(2\sqrt{\kappa})^{1-\alpha}\frac{1}{\text{Beta}(\frac{\lambda}{2\sqrt{\kappa}},1-\alpha)}\\
=&(2\sqrt{\kappa})^{1-\alpha}\frac{\lambda}{2\sqrt{\kappa}}\frac{\sin(\alpha\pi)}{\pi}\text{Beta}\left(\frac{\lambda}{2\sqrt{\kappa}}+1-\alpha,\alpha\right)\\
=&\frac{1}{\pi}\lambda(2\sqrt{\kappa})^{-\alpha}\sin(\alpha\pi)\int\limits_{0}^{1}y^{\frac{\lambda}{2\sqrt{\kappa}}-\alpha}(1-y)^{\alpha-1}\,\mathrm{d}y.
\end{align*}
Next, we change the variable $y$ by $e^{-2\sqrt{\kappa}u}$:
\begin{align*}
\Phi(\lambda)=&\lambda\cdot\frac{1}{\pi}(2\sqrt{\kappa})^{-\alpha}\sin(\alpha\pi)\int\limits_{0}^{\infty}e^{-\lambda u}e^{2\alpha\sqrt{\kappa}u}(1-e^{-2\sqrt{\kappa}u})^{\alpha-1}\cdot2\sqrt{\kappa}e^{-2\sqrt{\kappa}u}\,\mathrm{d}u\\
=& \lambda\cdot\frac{1}{\pi}\sin(\alpha\pi)(2\sqrt{\kappa})^{1-\alpha}\int\limits_{0}^{\infty}e^{-\lambda u}(e^{2\sqrt{\kappa}u}-1)^{\alpha-1}\,\mathrm{d}u.
\end{align*}
Thus,
$$\bar{\Pi}(t)=\frac{1}{\pi}\sin(\alpha\pi)(2\sqrt{\kappa})^{1-\alpha}(e^{2\sqrt{\kappa}t}-1)^{\alpha-1}.$$
Finally, we find $\Pi(\mathrm{d}t)$ by calculating the derivative of $\bar{\Pi}$:
$$\Pi(\mathrm{d}t)=\,\mathrm{d}t\cdot \frac{1}{\pi}(1-\alpha)\sin(\alpha\pi)e^{2\sqrt{\kappa}(\alpha-1)t}\left(\frac{2\sqrt{\kappa}}{1-e^{-2\sqrt{\kappa}t}}\right)^{2-\alpha}.$$

\paragraph{Acknowledgements.} 
The author thanks Yves Le Jan for useful suggestions and stimulating discussions, the author thanks Sophie Lemaire and the anonymous reviewer for careful reading and valuable comments. 

\bibliographystyle{amsalpha}
\bibliography{reference}
\end{document}